\documentclass[11pt,times,letter]{article}
\usepackage{amsmath, amsfonts,amsthm}
\usepackage{dsfont}
\usepackage{bbm}

\usepackage[english]{babel}
\usepackage[normalem]{ulem}

\usepackage{latexsym,amssymb,graphicx}
\usepackage{verbatim}
\usepackage{mathrsfs}
\usepackage{epsfig}

\usepackage{xcolor}
\usepackage{tikz}
\usepackage{pgfplots}
\usepgfplotslibrary{fillbetween}

\usepackage{url}
\usepackage{bm}
\usepackage[numbers]{natbib}

\usepackage[colorlinks,linkcolor=blue,citecolor=blue,anchorcolor=blue]{hyperref}

\usepackage{comment}

\newtheorem{theorem}{Theorem}[section]

\newtheorem{lemma}[theorem]{Lemma}
\newtheorem{proposition}[theorem]{Proposition}

\newtheorem{definition}[theorem]{Definition}
\newtheorem{remark}[theorem]{Remark}

\numberwithin{equation}{section}

\definecolor{Red}{rgb}{1.00, 0.00, 0.00}

\definecolor{DRed}{rgb}{0.5, 0.00, 0.00}

\definecolor{Blue}{rgb}{0.00, 0.00, 1.00}

\definecolor{Green}{rgb}{0.0, 0.4, 0.0}

\definecolor{Gray}{rgb}{0.7421875,0.7421875,0.7421875}

\definecolor{mycolor1}{rgb}{0.00000,0.44700,0.74100}%
\definecolor{mycolor2}{rgb}{0.85000,0.32500,0.09800}%
\definecolor{mycolor3}{rgb}{0.92900,0.69400,0.12500}%
\definecolor{mycolor4}{rgb}{0.49400,0.18400,0.55600}%
\definecolor{mycolor5}{rgb}{0.46600,0.67400,0.18800}%
\definecolor{mycolor6}{rgb}{0.00000,1.00000,1.00000}%

\allowdisplaybreaks[4]
\pgfplotsset{compat=newest}

\def\beq{\begin{eqnarray}}
\def\eeq{\end{eqnarray}}
\def\be*{\begin{eqnarray*}}
\def\ee*{\end{eqnarray*}}

\def \R{\mathbb{R}}

\def \eps {\varepsilon}

\def\R{\mathbb{R}}

\def\namedlabel#1#2{\begingroup
	#2%
	\def\@currentlabel{#2}%
	\phantomsection\label{#1}\endgroup
}

\def \1 {{\mathbf 1}}

\title{\Large \bf
Kyle meets time-inconsistency: a dynamic mean--variance informed trading problem
}
\date{\vspace{-5ex}}

\author{Shuoqing Deng
\thanks{
{\small The Hong Kong University of Science and Technology, Department of Mathematics (masdeng@ust.hk).  }}%
\and 
Zhichao Luo
\thanks{
{\small 
The Hong Kong University of Science and Technology, Department of Mathematics (zluoav@connect.ust.hk). }}%}
\and 
Zhenhua Wang
\thanks{
{\small 
Shandong University, Zhongtai Securities Institute for Financial Studies (zhenhuaw@sdu.edu.cn). 
}}%
}

\begin{document}

% \begin{center}
% {\Large Stabilization of Jump-Diffusion stochastic differential equation by hysteresis switching} 
% \end{center}
% \begin{center}
% {\large \today}\\[3cm] 
% \end{center}

\maketitle

\begin{abstract}

We study a Kyle--Back model in which the informed trader has dynamic mean--variance preferences, leading to a time-inconsistent trading problem that is formulated as an intrapersonal game. The model combines two equilibrium requirements: a trading--pricing equilibrium between the informed trader and the market maker, and a time-consistent equilibrium among the trader's successive selves. Following the framework of \cite{Cho}, we consider both risk-neutral and risk-averse informed traders. In the risk-neutral case, we derive the equilibrium trading strategy and price impact explicitly. In the risk-averse case, we establish existence and characterize the equilibrium through a coupled system of nonlinear ordinary differential equations. Our main technique is to reduce the equilibrium conditions to a forward--backward ODE system and resolve the resulting boundary conditions by a shooting argument. Economically, the mean--variance preference reshapes the intertemporal trade-off between exploiting current private information and preserving future informational advantage, generating a declining term structure of price impact and shifting both information revelation and informed-trading profits toward earlier stages of the trading horizon.

\vspace{0.6 cm}


\vspace{0.2 cm}
%\noindent{\textbf{JEL Classification}: D71, D81, G11, C62}
%\vspace{0.2 cm}


\end{abstract}

\section{Introduction}
%forget the strong equilibrium; forget the fixed point argument;
%Informed trading problem
The Kyle \cite{Kyle85} and Back \cite{Back92} model studies the equilibrium of trading--pricing with asymmetric information. In this model, there are three types of agents in the market: an informed trader, a market maker, and noise traders. A risky asset is traded in the market, whose intrinsic value is represented by a random variable $V$ and is privately observed only by the informed trader. The noise traders submit random orders. The market maker observes only the aggregate order flow, without being able to distinguish informed orders from noise trades, and uses this information to determine the asset price. In equilibrium, the informed trader chooses her trading strategy optimally given the pricing rule, while the market maker sets prices rationally based on the observed aggregate order flow.

The Kyle--Back model has been an important topic in market microstructure and has been extensively studied in the finance and mathematical finance literature; see, among others, \cite{ABO12,  BCG00, BCFLL18, BHMO12,  CalSta10, CampiCetin07, CampiCetinDan11, Cetin18, CetinDan16, CetinDan21, CKL23, CF16, HHL01, KX22}. In \cite{Cho}, the equilibrium is derived explicitly for both risk-neutral and risk-averse informed traders. In the risk-neutral case, the price impact is constant, whereas in the risk-averse case, it decreases over time and the price depends on the entire path of the order flow. In a series of papers \cite{BCEL21, BoseEkren21a, BoseEkren21b, CENV23, EkrenMZ}, Ekren and his coauthors developed a systematic approach to Kyle--Back models based on Brenier's theorem from optimal transport (OT). Using this approach, \cite{BCEL21} allows for a risk-averse market maker, while \cite{BoseEkren21a, BoseEkren21b} establish the existence of equilibrium in a general framework with a risk-averse informed trader and non-Gaussian beliefs. 
More recent developments include the FBSDE approach of \cite{QiaoZhang}, which characterizes continuous-time Kyle--Back equilibria beyond the usual Markovian or bridge-type constructions, and the model of \cite{MaXiaZhang2025}, which incorporates legal penalties and a large population of liquidity traders, thereby capturing the insider's trade-off between profitability and camouflage.

Our model combines the Kyle--Back informed-trading framework with a dynamic mean--variance preference. The latter introduces time inconsistency through the conditional variance term: a strategy preferred by the informed trader at the current time may no longer remain preferred when the problem is reconsidered at a later date. A standard approach to such problems is therefore to abandon global optimality and seek a time-consistent equilibrium among the decision maker's successive selves; for the development of this equilibrium approach to time-inconsistent stochastic control, see \cite{Strotz1955, BjorkKhapkoMurgoci2017, HuangZhou2021, HeJiang2021}, and the references therein. Mean--variance problems constitute one of the most important classes of applications of this framework. \cite{BasakChabakauri2010} derives an analytical time-consistent mean--variance portfolio strategy, while \cite{BjorkMurgociZhou2014} extends the analysis to state-dependent risk preferences. The effects of proportional transaction costs on equilibrium mean--variance portfolio choice are studied in \cite{DJJ24}. Different from these portfolio-selection problems, in the present paper the time-inconsistent mean--variance decision problem is embedded in a Kyle--Back equilibrium, where the informed trader's strategy simultaneously affects the information revealed to the market maker and hence the endogenous price process.

A particularly relevant comparison is with the emerging literature on time-inconsistent stochastic games, where equilibrium may involve two intertwined levels. In such problems, a player interacts not only with her own future selves at the intrapersonal level, but also with other players at the interpersonal level; examples include \cite{HuangZhou2022,LeiPun2024, ZhangZhou2025, Zhou2026}. This two-level structure is especially transparent in time-inconsistent mean-field games. 
\cite{LiangZhang2024} studies time-inconsistent mean-field and $n$-agent portfolio games under relative-performance criteria and explicitly constructs deterministic equilibria. 
\cite{BayraktarWang2025} investigates time-inconsistent mean-field games of Markov decision processes, distinguishes equilibrium notions according to the dependence of strategies on the population distribution, and establishes the existence of classical equilibria. 
More recently, \cite{BayraktarWangYuZhang2026} establishes the existence of relaxed equilibria for general time-inconsistent mean-field games via vanishing entropy regularization and also analyzes policy iteration for the regularized problem. 
In all these models, the mean-field equilibrium combines an intrapersonal equilibrium for the representative agent with an additional mean-field consistency requirement. Our Kyle--Back model exhibits an analogous but fundamentally different two-level structure. The first level is again an intrapersonal equilibrium among the informed trader's successive selves, generated here by the mean--variance criterion. The second level, however, is not an interpersonal game or a population fixed-point problem: it is the Kyle--Back rational-pricing condition, under which the market maker filters the fundamental value from the aggregate order flow generated by the trader's equilibrium strategy. Thus, the trader's extended HJB system must be coupled with an endogenous filtering problem so that the conjectured pricing rule is consistent with the information generated in equilibrium. To the best of our knowledge, the present paper is the first to introduce an explicitly time-inconsistent mean--variance preference into the Kyle--Back framework and to solve this coupled intrapersonal-equilibrium/rational-filtering problem. This coupling is also the source of the forward--backward structure and the shooting problem developed below.

%Risk-neutral case serves as a benchmark for the risk-averse case

Our main contribution are two folds.
The first is to identify and solve the two-level equilibrium structure generated by time inconsistency in the Kyle--Back model. The informed trader must satisfy an intrapersonal equilibrium condition across her successive selves, while her strategy must simultaneously be consistent with the market maker's rational filtering and pricing rule. This produces a coupling that is different from both standard time-inconsistent control and the mean-field consistency mechanisms appearing in time-inconsistent mean-field games.

The second contribution is technical. Under a Gaussian--linear ansatz, the coupled equilibrium conditions reduce to a forward--backward ODE system with singular terminal behavior, which we solve by a shooting construction. This yields an explicit solution without CARA risk aversion and, in the risk-averse case, a global characterization through a uniquely determined shooting parameter; see, e.g., \cite{Keller1971, RobertsShipman1972} for the classical shooting philosophy.

Our equilibrium characterization also yields several economic implications. The informed trader faces a dynamic trade-off between exploiting the current pricing bias and preserving her informational advantage for future trading. Although this trade-off already exists in the classical Kyle--Back model, CARA risk aversion and the mean--variance preference change its intertemporal balance and generate a declining term structure of equilibrium price impact. 
Compared with the classical risk-neutral case, both preferences increase price impact early in the trading horizon and reduce it near maturity, making the market shallower initially and deeper later. Since the rate of information revelation is directly linked to price impact through the evolution of the conditional variance, this front-loading of price impact leads to earlier revelation of private information and an earlier realization of informed-trading profits. The mean--variance component has an additional interpretation: because successive selves evaluate conditional wealth variance differently, the current informed trader must anticipate future re-optimization when deciding how rapidly to exploit and reveal her private information. These implications and their sensitivity to the preference parameters are examined quantitatively in Section~\ref{sec:financial}.

\textbf{Organization of the paper.}
Section~\ref{sec:model} introduces the market model, formulates the informed trader's time-inconsistent problem and the market maker's filtering and pricing problem, and defines the equilibrium. Section~\ref{sec:main} derives the equilibrium conditions and presents our main results. Section~\ref{sec:financial} discusses the economic implications of the equilibrium, with particular emphasis on price impact, information revelation, and the effects of risk preferences. Section~\ref{sec:proof} contains the proofs of the main results.

\section{Market model and problem formulation}
\label{sec:model}

We first introduce an informed-trading model similar to that of \cite{Cho}. Consider a continuous-time market over a finite horizon $[0,T]$, with $T>0$. The market consists of a market maker, an informed trader, and noise traders. The market maker sets the price, the informed trader trades using private information, and the noise traders trades randomly. There is one risky asset in the market, and the risk-free interest rate is normalized to zero. At time $T$, the risky asset pays a random value $V$, which is known to the informed trader from time $0$ but not observable to the market maker.

Let $(\Omega,\mathcal{F},\mathbb{P})$ be a probability space supporting $V$ and a standard Brownian motion $B$ independent of $V$. We assume that $V\sim\mathcal{N}(\mu,\Sigma_0)$, where $\Sigma_0>0$. The cumulative order submitted by the noise traders is modeled as $Z_t=\sigma B_t$, where $\sigma>0$. The informed trader submits orders continuously, with $\alpha_t$ denoting the trading rate. Consequently, the cumulative informed order is $X_t=\int_0^t\alpha_s\,\mathrm{d}s$. The aggregate order flow is $Y_t:=X_t+Z_t$, and hence
\begin{equation}
    \mathrm{d}Y_t
    =
    \alpha_t\,\mathrm{d}t
    +
    \sigma\,\mathrm{d}B_t,
    \qquad
    Y_0=0.
    \label{eq:aggregate_order}
\end{equation}
The market maker observes $Y_t$, but cannot distinguish between $X_t$ and $Z_t$. Thus, the informed order is observed only through the noise generated by $Z$. Let $\mathbb{F}^Y=(\mathcal{F}^Y_t)_{0\leq t\leq T}$ denote the filtration generated by $Y$, which represents the information available to the market maker. The informed trader additionally observes $V$, hence the informed trader's filtration is $\mathcal{F}^I_t=\mathcal{F}^Y_t\vee\sigma(V)$, where $\sigma(V)$ denotes the $\sigma$-algebra generated by $V$.

We consider pricing rules of the form
\begin{equation}
    P_t=H(t,\xi_t),
    \qquad
    \xi_t=\int_0^t\lambda(s)\,\mathrm{d}Y_s,
    \label{eq:pricing_rule}
\end{equation}
where $H:[0,T]\times\mathbb{R}\to\mathbb{R}$ and $\lambda:[0,T)\to(0,\infty)$ are deterministic. The process $\xi$ summarizes the information extracted from the aggregate order flow, while $\lambda$ measures the instantaneous price pressure. We assume that $H$ is continuous on $[0,T]\times\mathbb{R}$, belongs to $C^{1,2}([0,T)\times\mathbb{R})$, and is strictly increasing in its second argument. The pair $(H,\lambda)$ is called a pricing rule.

\subsection{A risk-averse and time-inconsistent informed trader}
\label{subsec:insider_problem}

Fix a pricing rule $(H,\lambda)$ and condition on $V=v$, where $v$ denotes a realization of $V$. Starting from time $t$ with current wealth $w$, the terminal wealth generated by a trading strategy $\alpha$ is
\[
    W_T^\alpha
    =
    w+\int_t^T
    \bigl(v-H(s,\xi_s^\alpha)\bigr)\alpha_s\,\mathrm{d}s.
\]
For a random variable $\mathcal{X}$ with the required conditional exponential moment, define
\begin{equation}
    \operatorname{CE}_{\rho,t,\xi,w}(\mathcal{X})
    :=
    \begin{cases}
        -\dfrac{1}{\rho}
        \log
        \mathbb{E}_{t,\xi,w}
        \bigl[e^{-\rho\mathcal{X}}\bigr],
        & \rho>0,\\[0.8em]
        \mathbb{E}_{t,\xi,w}[\mathcal{X}],
        & \rho=0.
    \end{cases}
    \label{eq:conditional_certainty_equivalent}
\end{equation}
At time $t$, the informed trader evaluates the strategy through the criterion
\begin{equation}
    J(t,\xi,w;v,\alpha)
    =
    \operatorname{CE}_{\rho,t,\xi,w}
    \bigl(W_T^\alpha\bigr)
    -
    \frac{\eta}{2}
    \operatorname{Var}_{t,\xi,w}
    \bigl(W_T^\alpha\bigr),
    \qquad
    \rho,\eta\geq0.
    \label{eq:risk_averse_objective}
\end{equation}
For $\rho>0$, the first term is the conditional certainty equivalent associated with the constant absolute risk aversion (CARA) exponential utility $\mathscr{U}_\rho(x)=-e^{-\rho x}$. Maximizing this certainty is equivalent to maximizing conditional expected exponential utility, and $\rho$ measures the degree of risk aversion. Its translation invariance will allow current wealth to be separated from the continuation problem in Subsection~\ref{subsec:extended_hjb}.

The second term in \eqref{eq:risk_averse_objective} introduces an additional penalty on conditional wealth variance. When $\eta>0$, this term makes the criterion time-inconsistent and therefore prevents the direct application of the usual dynamic programming principle: a strategy preferred at one time may cease to be preferred at a later time. By following \cite{BjorkKhapkoMurgoci2017}, we formulate the problem as an intrapersonal game among the trader's successive selves.

\begin{remark}[Benchmark cases]
\label{rem:benchmark_cases}
The criterion \eqref{eq:risk_averse_objective} contains two important benchmark models.
\begin{enumerate}
    \item[\textnormal{(i)}]
    If $\eta=0$ and $\rho>0$, maximizing \eqref{eq:risk_averse_objective} is equivalent to maximizing conditional expected exponential utility. The informed trader's problem then reduces to the risk-averse Kyle--Back model studied by \cite{Cho}.

    \item[\textnormal{(ii)}]
    As $\rho\downarrow0$, a cumulant expansion yields
    \[
        -\frac{1}{\rho}
        \log
        \mathbb{E}_{t,\xi,w}
        \left[
            e^{-\rho W_T^\alpha}
        \right]
        =
        \mathbb{E}_{t,\xi,w}
        \bigl[W_T^\alpha\bigr]
        -
        \frac{\rho}{2}
        \operatorname{Var}_{t,\xi,w}
        \bigl(W_T^\alpha\bigr)
        +
        o(\rho),
    \]
    whenever the required exponential moments exist. This is consistent with the definition of $\operatorname{CE}_{\rho,t,\xi,w}$ at $\rho=0$. In particular, when $\rho=0$ and $\eta>0$, the criterion becomes
    \[
        \mathbb{E}_{t,\xi,w}
        \bigl[W_T^\alpha\bigr]
        -
        \frac{\eta}{2}
        \operatorname{Var}_{t,\xi,w}
        \bigl(W_T^\alpha\bigr),
    \]
    which is the time-inconsistent mean--variance benchmark considered separately below. With a slight abuse of terminology, we will  refer to this as the risk-neutral case. Here, ``risk-neutral'' refers only to the absence of CARA risk aversion; the informed trader remains averse to conditional wealth variance through the mean--variance term.
\end{enumerate}
The case $\rho>0$ and $\eta>0$, in which CARA risk aversion and time inconsistency are both present, is the main focus of this paper.
\end{remark}

We next introduce the admissible class of trading strategies. For a given pair $$(H,\lambda)$$, a trading strategy $\alpha$ is said to be admissible if it is $\mathbb{F}^I$-progressively measurable and locally square-integrable on $[0,T)$, and if
$$
\int_0^T |\alpha_t|,\mathrm{d}t<\infty\;\text{a.s.},\quad \mathbb{E}\big[|W_T^\alpha|^2\big]<\infty.
$$
When $\rho>0$, we additionally require
$
\mathbb{E}\big[e^{-\rho W_T^\alpha}\big]<\infty,
$
so that the conditional certainty equivalent in \eqref{eq:conditional_certainty_equivalent} is well defined.
%We next specify the admissibility requirements. A trading strategy $\alpha$ is admissible for $(H,\lambda)$ if it is $\mathbb{F}^I$-progressively measurable, locally square-integrable on $[0,T)$, and satisfies $\int_0^T|\alpha_t|\,\mathrm{d}t<\infty$ almost surely and $\mathbb{E}[|W_T^\alpha|^2]<\infty$, together with $\mathbb{E}[e^{-\rho W_T^\alpha}]<\infty$ when $\rho>0$. 
We denote the resulting admissible set by $\mathcal{A}(H,\lambda)$. In addition, $\mathcal{A}(H,\lambda)$ is assumed to be stable under the bounded constant pasting deviations used in the equilibrium definition below.
%The admissible class, denoted by $\mathcal{A}(H,\lambda)$, is required to be stable under the bounded constant deviation variations introduced below. 
The local square-integrability requirement allows for trading rates that may become singular as $t\to T+$, while the pathwise integrability condition guarantees that the cumulative order remains well defined up to the terminal time.
%Local square-integrability accommodates trading rates that may become singular near $T$ while still ensuring a well-defined cumulative order and terminal wealth. The additional exponential-integrability condition ensures that the certainty equivalent in \eqref{eq:conditional_certainty_equivalent} is well defined when $\rho>0$.

Let $\widehat{\alpha}$ be a candidate feedback strategy. For $t<T$, a bounded constant action $a\in\R$, and sufficiently small $\varepsilon>0$, define the pasted strategy
\[
    a\otimes_{t,\eps}\widehat\alpha
    :=
    \begin{cases}
        a, & t\leq s<t+\varepsilon,\\
        \widehat{\alpha}(s,\xi_s^{t,\varepsilon,a};v),
        & t+\varepsilon\leq s<T.
    \end{cases}
\]
Thus, $a\otimes_{t,\varepsilon}\widehat{\alpha}$ deviates from $\widehat{\alpha}$ by using the constant trading rate $a$ on $[t,t+\varepsilon)$ and then reverts to $\widehat{\alpha}$ thereafter. We refer to $a\otimes_{t,\varepsilon}\widehat{\alpha}$ as a bounded constant pasting deviation from $\widehat{\alpha}$.

\begin{definition}[Equilibrium strategy of the informed trader]
\label{def:insider_equilibrium}
For a fixed pricing rule $(H,\lambda)$, an admissible feedback strategy $\widehat{\alpha}$ is called a weak equilibrium strategy if for every $t<T$, every admissible state $(\xi,w,v)$, and every bounded constant action $a\in\R$,
\begin{equation}
    \limsup_{\varepsilon\downarrow0}
    \frac{
        J(t,\xi,w;v,a\otimes_{t,\eps}\widehat\alpha )
        -
        J(t,\xi,w;v,\widehat{\alpha})
    }{\varepsilon}
    \leq0.
    \label{eq:insider_equilibrium_condition}
\end{equation}
%and every admissible bounded constant deviation $a$.
\end{definition}

\subsection{Market maker's problem: filtering and pricing}
\label{subsec:market_maker_problem}

The market maker is risk-neutral and sets the price according to the information contained in the aggregate order flow. For a trading strategy $\alpha$, a pricing rule $(H,\lambda)$ is \textit{rational} if
\begin{equation}
    H(t,\xi_t)
    =
    \mathbb{E}\left[V\mid\mathcal{F}^Y_t\right],
    \qquad
    0\leq t<T.
    \label{eq:rational_pricing}
\end{equation}
The market maker's inference problem is endogenous: the information contained in the order flow depends on the informed trader's strategy, while the strategy itself depends on the price inferred from that order flow.

Let $\Sigma(t):=\operatorname{Var}(V\mid\mathcal{F}^Y_t)$ denote the posterior variance of the fundamental value. The filtering equations governing the conditional mean and variance will be derived in Section~\ref{sec:main}. We require full revelation at the terminal time, $P_T=V$ $\mathbb{P}$-a.s., equivalently, $\Sigma(T)=0$. This is an equilibrium-path condition and therefore imposes no restriction on terminal states outside the support of the equilibrium state process.

\subsection{Intrapersonal and market equilibrium}
\label{subsec:equilibrium_definition}

The model naturally combines two levels of equilibrium requirements. On the individual level, the informed trader's strategy should satisfy the intrapersonal equilibrium condition in Definition \ref{def:insider_equilibrium} under the corresponding pricing rule. On the market level, the informed trader and the market maker should reach a pricing--trading equilibrium: the price must be rational relative to the order flow induced by the informed trader's strategy.

\begin{definition}[Market equilibrium]
\label{def:market_equilibrium}
A pair $\bigl((H^*,\lambda),\widehat{\alpha}^*\bigr)$ is called a weak market equilibrium if $\widehat{\alpha}^*$ is a weak equilibrium strategy under $(H^*,\lambda)$ in the sense of Definition~\ref{def:insider_equilibrium}, and $(H^*,\lambda)$ satisfies the rational pricing condition for the order flow $Y^*$ generated by $\widehat{\alpha}^*$:
\[
    H^*(t,\xi_t^*)
    =
    \mathbb{E}
    \left[
        V\mid\mathcal{F}^{Y^*}_t
    \right],
    \qquad
    0\leq t<T,
\]
where
\[
    \mathrm{d}Y_t^*
    =
    \widehat{\alpha}_t^*\,\mathrm{d}t
    +
    \sigma\,\mathrm{d}B_t,
    \qquad
    \xi_t^*
    =
    \int_0^t
    \lambda(s)\,\mathrm{d}Y_s^*.
\]
The equilibrium must also satisfy $H^*(T,\xi_T^*)=V$ almost surely.
\end{definition}

The two equilibrium requirements are coupled. The extended HJB system characterizes the informed trader's response to a fixed pricing rule, while rational filtering determines whether that pricing rule is consistent with the resulting order flow. Their joint characterization is developed in Section~\ref{sec:main}.

\section{Main results}
\label{sec:main}

\subsection{Informed trader's problem: extended HJB}
\label{subsec:extended_hjb}

We now fix a pricing rule $(H,\lambda)$ and derive the extended HJB system for the informed trader. Throughout this subsection, we condition on $V=v$ and denote the mispricing by $z(t,\xi;v):=v-H(t,\xi)$. Under a trading rate $\alpha$, the state and wealth processes satisfy $\mathrm{d}\xi_t=\lambda(t)\alpha_t\,\mathrm{d}t+\sigma\lambda(t)\,\mathrm{d}B_t$ and $\mathrm{d}W_t=z(t,\xi_t;v)\alpha_t\,\mathrm{d}t$, respectively. %The wealth process has no diffusion term; the uncertainty in terminal wealth is transmitted through the price state $\xi$. 
For a smooth function $\varphi(t,\xi,w)$ and an action $a\in\mathbb{R}$, define the generator
\begin{equation}
    \mathcal{L}^a\varphi
    :=
    \varphi_t
    +
    \lambda(t)a\varphi_\xi
    +
    z(t,\xi;v)a\varphi_w
    +
    \frac{1}{2}\sigma^2\lambda(t)^2\varphi_{\xi\xi}.
    \label{eq:insider_generator}
\end{equation}

The integral $\int_t^Tz(s,\xi_s^\alpha;v)\alpha_s\,\mathrm{d}s$ represents the total running profit for the informed trader, and 
%Let$R_{t,T}^{\alpha}:=\int_t^Tz(s,\xi_s^\alpha;v)\alpha_s\,\mathrm{d}s$ denote the continuation profit, so that 
$W_T^\alpha=w+\int_t^Tz(s,\xi_s^\alpha;v)\alpha_s\,\mathrm{d}s$. Let $\widehat{\alpha}$ be a weak equilibrium feedback strategy. We define its conditional certainty equivalent and conditional mean by
\begin{equation}
    h(t,\xi;v)
    :=
    \operatorname{CE}_{\rho,t,\xi}
    \left(\int_t^Tz(s,\xi_s^{\widehat\alpha};v)\widehat\alpha_s\,\mathrm{d}s\right),
    \qquad
    m(t,\xi;v)
    :=
    \mathbb{E}_{t,\xi}
    \left[ \int_t^Tz(s,\xi_s^{\widehat\alpha};v)\widehat\alpha_s\,\mathrm{d}s  \right],
    \label{eq:continuation_h_m}
\end{equation}
and the equilibrium continuation value by
\begin{equation}
    u(t,\xi;v)
    :=
    h(t,\xi;v)
    -
    \frac{\eta}{2}
    \operatorname{Var}_{t,\xi}
    \left( \int_t^Tz(s,\xi_s^{\widehat\alpha};v)\widehat\alpha_s\,\mathrm{d}s \right).
    \label{eq:continuation_u}
\end{equation}
Since current wealth is known at time $t$ and enters terminal wealth additively, the translation invariance of the conditional certainty equivalent and conditional variance yields $J(t,\xi,w;v,\widehat{\alpha})=w+u(t,\xi;v)$. We shall also use $g(t,\xi,w;v):=w+m(t,\xi;v)$ and $U(t,\xi,w;v):=w+u(t,\xi;v)$.

We first derive the auxiliary equations satisfied by $h$ and $m$ under the equilibrium continuation strategy. Suppose $\rho>0$ and define $\Phi(t,\xi,w;v):=e^{-\rho(w+h(t,\xi;v))}$. The definition of $h$ in \eqref{eq:continuation_h_m} implies
\[
    \Phi(t,\xi,w;v)
    =
    \mathbb{E}_{t,\xi,w}
    \left[
        e^{-\rho W_T^{\widehat{\alpha}}}
    \right],
\]
and hence $\mathcal{L}^{\widehat{\alpha}}\Phi=0$. Substituting the exponential representation of $\Phi$ and dividing by $-\rho\Phi$ imply
\begin{equation}
    0
    =
    h_t
    +
    \frac{1}{2}\sigma^2\lambda(t)^2h_{\xi\xi}
    -
    \frac{\rho}{2}\sigma^2\lambda(t)^2h_\xi^2
    +
    \widehat{\alpha}(t,\xi;v)
    \bigl(
        \lambda(t)h_\xi
        +
        v-H(t,\xi)
    \bigr).
    \label{eq:auxiliary_h}
\end{equation}
%The continuation mean 
Moreover, $m$ in \eqref{eq:continuation_h_m} satisfies the Kolmogorov equation
\begin{equation}
    0
    =
    m_t
    +
    \frac{1}{2}\sigma^2\lambda(t)^2m_{\xi\xi}
    +
    \widehat{\alpha}(t,\xi;v)
    \bigl(
        \lambda(t)m_\xi
        +
        v-H(t,\xi)
    \bigr).
    \label{eq:auxiliary_m}
\end{equation}

%Following the game-theoretic approach, we view the problem as an intrapersonal game among the trader's successive selves. 
The extended HJB equation associated with the criterion \eqref{eq:risk_averse_objective} is
\begin{equation}
    0
    =
    \sup_{a\in\mathbb{R}}
    \left\{
        u_t
        +
        \frac{1}{2}\sigma^2\lambda(t)^2u_{\xi\xi}
        -
        \frac{\rho}{2}\sigma^2\lambda(t)^2h_\xi^2
        -
        \frac{\eta}{2}\sigma^2\lambda(t)^2m_\xi^2
        +
        a\bigl(
            \lambda(t)u_\xi
            +
            v-H(t,\xi)
        \bigr)
    \right\}.
    \label{eq:extended_hjb_general}
\end{equation}
The term involving $h_\xi^2$ arises from the exponential certainty equivalent, while the term involving $m_\xi^2$ is the correction induced by the time inconsistency of the conditional variance penalty.

Combining \eqref{eq:auxiliary_h}, \eqref{eq:auxiliary_m}, and \eqref{eq:extended_hjb_general}, we obtain the extended HJB system
\begin{equation}
    \left\{
    \begin{aligned}
        0
        &=
        \sup_{a\in\mathbb{R}}
        \left\{
            u_t
            +
            \frac{1}{2}\sigma^2\lambda(t)^2u_{\xi\xi}
            -
            \frac{\rho}{2}\sigma^2\lambda(t)^2h_\xi^2
            -
            \frac{\eta}{2}\sigma^2\lambda(t)^2m_\xi^2
            +
            a\bigl(
                \lambda(t)u_\xi
                +
                v-H(t,\xi)
            \bigr)
        \right\},\\
        0
        &=
        h_t
        +
        \frac{1}{2}\sigma^2\lambda(t)^2h_{\xi\xi}
        -
        \frac{\rho}{2}\sigma^2\lambda(t)^2h_\xi^2
        +
        \widehat{\alpha}(t,\xi;v)
        \bigl(
            \lambda(t)h_\xi
            +
            v-H(t,\xi)
        \bigr),\\
        0
        &=
        m_t
        +
        \frac{1}{2}\sigma^2\lambda(t)^2m_{\xi\xi}
        +
        \widehat{\alpha}(t,\xi;v)
        \bigl(
            \lambda(t)m_\xi
            +
            v-H(t,\xi)
        \bigr).
    \end{aligned}
    \right.
    \label{eq:extended_hjb_system}
\end{equation}

Since the control set is unbounded and the Hamiltonian in the first equation of \eqref{eq:extended_hjb_system} is linear in $a$, the supremum is finite only if
\begin{equation}
    \lambda(t)u_\xi(t,\xi;v)  +v-H(t,\xi)= 0.
    \label{eq:singular_constraint}
\end{equation}
The extended HJB system then reduces to
\begin{equation}
    \left\{
    \begin{aligned}
        0
        &=
        \lambda(t)u_\xi+v-H(t,\xi),\\
        0
        &=
        u_t
        +
        \frac{1}{2}\sigma^2\lambda(t)^2u_{\xi\xi}
        -
        \frac{\rho}{2}\sigma^2\lambda(t)^2h_\xi^2
        -
        \frac{\eta}{2}\sigma^2\lambda(t)^2m_\xi^2,\\
        0
        &=
        h_t
        +
        \frac{1}{2}\sigma^2\lambda(t)^2h_{\xi\xi}
        -
        \frac{\rho}{2}\sigma^2\lambda(t)^2h_\xi^2
        +
        \widehat{\alpha}(t,\xi;v)
        \bigl(
            \lambda(t)h_\xi
            +
            v-H(t,\xi)
        \bigr),\\
        0
        &=
        m_t
        +
        \frac{1}{2}\sigma^2\lambda(t)^2m_{\xi\xi}
        +
        \widehat{\alpha}(t,\xi;v)
        \bigl(
            \lambda(t)m_\xi
            +
            v-H(t,\xi)
        \bigr),
    \end{aligned}
    \right.
    \label{eq:reduced_extended_hjb}
\end{equation}
with terminal conditions
\begin{equation}
    h(T,\xi_T;v)
    =
    m(T,\xi_T;v)
    =
    u(T,\xi_T;v)
    =
    0.
    \label{eq:terminal_continuation_values}
\end{equation}
Condition \eqref{eq:singular_constraint} eliminates the dependence of the first equation on the action variable $a$. The candidate trading rate is therefore not determined by pointwise maximization. Later we shall see it is determined jointly by the auxiliary equations, rational pricing, and $\Sigma(T)=0$.

\noindent\textbf{The risk-neutral case.}
Throughout the paper, we will highlight the so-called \textit{risk-neutral} case, by which we mean the absence of CARA risk aversion, i.e. $\rho=0$. This benchmark is particularly useful because it admits an explicit equilibrium solution, making the coupling between the trader's intrapersonal equilibrium condition and the market maker's filtering condition transparent. It also provides a natural benchmark for the shooting construction used later in the risk-averse case.

In this case, the conditional certainty equivalent reduces to the conditional mean, so $h=m$. The auxiliary equation for $h$ then coincides with \eqref{eq:auxiliary_m}, and %the separate construction of $h$ is unnecessary. 
the extended HJB system becomes
\begin{equation}
    \left\{
    \begin{aligned}
        0
        &=
        \sup_{a\in\mathbb{R}}
        \left\{
            u_t
            +
            \frac{1}{2}\sigma^2\lambda(t)^2u_{\xi\xi}
            -
            \frac{\eta}{2}\sigma^2\lambda(t)^2m_\xi^2
            +
            a\bigl(
                \lambda(t)u_\xi
                +
                v-H(t,\xi)
            \bigr)
        \right\},\\
        0
        &=
        m_t
        +
        \frac{1}{2}\sigma^2\lambda(t)^2m_{\xi\xi}
        +
        \widehat{\alpha}(t,\xi;v)
        \bigl(
            \lambda(t)m_\xi
            +
            v-H(t,\xi)
        \bigr).
    \end{aligned}
    \right.
    \label{eq:risk_neutral_extended_hjb_system}
\end{equation}
Since the Hamiltonian remains linear in $a$, the singular constraint \eqref{eq:singular_constraint} is unchanged. Hence, \eqref{eq:risk_neutral_extended_hjb_system} reduces to
\begin{equation}
    \left\{
    \begin{aligned}
        0
        &=
        \lambda(t)u_\xi+v-H(t,\xi),\\
        0
        &=
        u_t
        +
        \frac{1}{2}\sigma^2\lambda(t)^2u_{\xi\xi}
        -
        \frac{\eta}{2}\sigma^2\lambda(t)^2m_\xi^2,\\
        0
        &=
        m_t
        +
        \frac{1}{2}\sigma^2\lambda(t)^2m_{\xi\xi}
        +
        \widehat{\alpha}(t,\xi;v)
        \bigl(
            \lambda(t)m_\xi
            +
            v-H(t,\xi)
        \bigr),
    \end{aligned}
    \right.
    \label{eq:risk_neutral_reduced_extended_hjb}
\end{equation}
with the terminal conditions
\begin{equation}
    m(T,\xi_T;v)=u(T,\xi_T;v)=0.
    \label{eq:risk_neutral_terminal_continuation_values}
\end{equation}
Thus, the risk-neutral construction uses the two continuation functions $m$ and $u$, whereas the risk-averse construction additionally requires the certainty-equivalent function $h$.

\subsection{Gaussian--linear ansatz}
\label{subsec:gaussian_linear_ansatz}

Under the Gaussian prior specified in Section~\ref{sec:model}, we seek an equilibrium in which both the pricing rule and the informed trading strategy are linear. We conjecture that $H(t,\xi)=\mu+q\xi$ for some constant $q>0$, where the positivity of $q$ follows from the strict monotonicity of $H$ with respect to $\xi$.

The representation of the price is invariant under the rescaling
\[
    (q,\lambda,\xi)
    \longmapsto
    \left(\frac{q}{c},c\lambda,c\xi\right),
    \qquad c>0.
\]
Indeed, replacing $\lambda$ by $c\lambda$ scales $\xi$ by the same factor, while replacing $q$ by $q/c$ leaves $H(t,\xi_t)$ unchanged. Taking $c=q$, we normalize $q=1$. By relabeling the rescaled state and coefficient as $\xi$ and $\lambda$ respectively, we obtain $H(t,\xi)=\mu+\xi$ and
\begin{equation}
    \mathrm{d}H(t,\xi_t)
    =
    \lambda(t)\,\mathrm{d}Y_t.
    \label{eq:linear_price_dynamics}
\end{equation}
Under this normalization, $\lambda$ represents the price impact associated with the aggregate order flow.

For a fixed realization $V=v$, define the pricing error by $e(t,\xi;v):=v-H(t,\xi)=v-\mu-\xi$. We write $e$ when its arguments are clear. We conjecture that the equilibrium feedback strategy and the continuation functions take the forms
\begin{equation}
    \widehat{\alpha}(t,e)
    =
    k(t)e,
    \qquad
    m(t,e)
    =
    A(t)e^2+C(t),
    \qquad
    h(t,e)
    =
    \Gamma(t)e^2+E(t),
    \label{eq:linear_quadratic_ansatz}
\end{equation}
where $k,A,\Gamma,C$, and $E$ are deterministic functions of time. Integrating the singular constraint \eqref{eq:singular_constraint} under the normalization $H(t,\xi)=\mu+\xi$ yields
\begin{equation}
    u(t,e)
    =
    \frac{e^2}{2\lambda(t)}
    +
    D(t),
    \label{eq:u_ansatz}
\end{equation}
where $D$ is deterministic.

Now we derive formulas satisfied by $A, C, \gamma, E, D, \lambda, k$. By first substituting the quadratic representation of $m$ in \eqref{eq:linear_quadratic_ansatz} into \eqref{eq:auxiliary_m}, then combining with $m_\xi=-2Ae$ and $m_{\xi\xi}=2A$, we can matching the coefficients of $e^2$ and the constant terms to obtain that
\begin{equation}
    A'(t)
    +
    k(t)\bigl(1-2\lambda(t)A(t)\bigr)
    =
    0,
    \qquad
    C'(t)
    +
    \sigma^2\lambda(t)^2A(t)
    =
    0.
    \label{eq:A_C_equations}
\end{equation}
Similarly, substituting the quadratic representation of $h$ in \eqref{eq:linear_quadratic_ansatz} into \eqref{eq:auxiliary_h}  and combining with $h_\xi=-2\Gamma e$ and $h_{\xi\xi}=2\Gamma$, then matching the coefficients yields
\begin{equation}
    \Gamma'(t)
    +
    k(t)\bigl(1-2\lambda(t)\Gamma(t)\bigr)
    -
    2\rho\sigma^2\lambda(t)^2\Gamma(t)^2
    =
    0,
    \qquad
    E'(t)
    +
    \sigma^2\lambda(t)^2\Gamma(t)
    =
    0.
    \label{eq:B_E_equations}
\end{equation}
The additional quadratic term in the equation for $\Gamma$ reflects the CARA risk aversion of the informed trader.
Finally, substituting \eqref{eq:linear_quadratic_ansatz} and \eqref{eq:u_ansatz} into the second equation of \eqref{eq:reduced_extended_hjb} yields
\begin{equation}
    \lambda'(t)
    =
    -4\sigma^2\lambda(t)^4
    \bigl(
        \rho \Gamma(t)^2+\eta A(t)^2
    \bigr),
    \qquad
    D'(t)
    +
    \frac{1}{2}\sigma^2\lambda(t)
    =
    0.
    \label{eq:lambda_D_equations}
\end{equation}
Notice that the two risk adjustments $\rho$ and $\eta$ enter the dynamics of price impact through $\rho \Gamma^2+\eta A^2$.

Along the equilibrium path, terminal revelation implies
\[
    e(T,\xi_T;v)
    =
    v-H(T,\xi_T)
    =
    0.
\]
Consequently, the terminal relations in \eqref{eq:terminal_continuation_values} reduce to
\begin{equation}
    C(T)=E(T)=D(T)=0.
    \label{eq:constant_terminal_conditions}
\end{equation}
No terminal conditions for $A$ and $\Gamma$ follow directly from \eqref{eq:terminal_continuation_values}, since their quadratic terms vanish along the equilibrium path.

The functions $C,E$, and $D$ do not affect the candidate trading strategy or the price impact. The remaining functions $A,\Gamma,\lambda$, and $k$ are determined by combining \eqref{eq:A_C_equations}--\eqref{eq:lambda_D_equations} with the rational pricing condition.

\noindent\textbf{The risk-neutral case.} When $\rho=0$, the identity $h=m$ makes the separate quadratic representation of $h$ unnecessary. We therefore use
\begin{equation}
    \widehat{\alpha}(t,e)=k(t)e,
    \qquad
    m(t,e)=A(t)e^2+C(t),
    \qquad
    u(t,e)=\frac{e^2}{2\lambda(t)}+D(t).
    \label{eq:risk_neutral_linear_quadratic_ansatz}
\end{equation}
Substituting \eqref{eq:risk_neutral_linear_quadratic_ansatz} into \eqref{eq:risk_neutral_reduced_extended_hjb} and matching the quadratic and constant terms yield
\begin{equation}
    \left\{
    \begin{aligned}
        0
        &=
        A'(t)
        +
        k(t)\bigl(1-2\lambda(t)A(t)\bigr),\\
        0
        &=
        C'(t)
        +
        \sigma^2\lambda(t)^2A(t),\\
        0
        &=
        \lambda'(t)
        +
        4\eta\sigma^2\lambda(t)^4A(t)^2,\\
        0
        &=
        D'(t)
        +
        \frac{1}{2}\sigma^2\lambda(t).
    \end{aligned}
    \right.
    \label{eq:risk_neutral_coefficient_equations}
\end{equation}
Terminal revelation and \eqref{eq:risk_neutral_terminal_continuation_values} imply
\begin{equation}
    C(T)=D(T)=0.
    \label{eq:risk_neutral_constant_terminal_conditions}
\end{equation}
We can see that the risk-neutral coefficient system consists of $A,C,\lambda,D$, and $k$. The additional functions $\Gamma$ and $E$ are needed only when $\rho>0$.

\subsection{Rational pricing and the coupled system}
\label{subsec:rational_pricing_system}

It remains to impose the rational pricing condition on the Gaussian--linear candidate. By substituting the representation of $\widehat \alpha$ in \eqref{eq:linear_quadratic_ansatz} into \eqref{eq:aggregate_order}, the dynamics of the aggregate order flow are
\begin{equation}
    \mathrm{d}Y_t
    =
    k(t)(V-P_t)\,\mathrm{d}t
    +
    \sigma\,\mathrm{d}B_t.
    \label{eq:linear_order_flow}
\end{equation}
Since $V$ is Gaussian and the observation dynamics are linear with Gaussian noise, the conditional distribution of $V$ given the order flow remains Gaussian.

Let $\widehat{V}_t:=\mathbb{E}[V\mid\mathcal{F}_t^Y]$ and $\Sigma(t):=\operatorname{Var}(V\mid\mathcal{F}_t^Y)$ denote the conditional mean and conditional variance, respectively. The rational pricing condition requires $P_t=\widehat{V}_t$. Moreover,
\[
    \mathbb{E}
    \left[
        \widehat{\alpha}_t
        \mid\mathcal{F}_t^Y
    \right]
    =
    k(t)\bigl(\widehat{V}_t-P_t\bigr)
    =
    0.
\]
Thus, the informed trader's strategy is inconspicuous from the perspective of the market maker.

The Kalman--Bucy filter associated with \eqref{eq:linear_order_flow} yields
\begin{equation}
    \mathrm{d}\widehat{V}_t
    =
    \frac{k(t)\Sigma(t)}{\sigma^2}\,\mathrm{d}Y_t,
    \qquad
    \Sigma'(t)
    =
    -\frac{k(t)^2}{\sigma^2}\Sigma(t)^2.
    \label{eq:filtering_equations}
\end{equation}
%Comparing the conditional-mean dynamics with
Combining this with $\widehat V_t =P_t$ and $\mathrm{d}P_t=\lambda(t)\,\mathrm{d}Y_t$ implies $\lambda(t)=k(t)\Sigma(t)/\sigma^2$, or equivalently
\begin{equation} 
    k(t)
    =
    \frac{\sigma^2\lambda(t)}{\Sigma(t)}.
    \label{eq:filtering_consistency}
\end{equation}
Substitution into the second equality in \eqref{eq:filtering_equations} yields
\begin{equation}
    \Sigma'(t)=-\sigma^2\lambda(t)^2,\quad 
    \frac{\Sigma'(t)}{\Sigma(t)}=-\lambda(t)k(t).
    \label{eq:variance_equation}
\end{equation}
The prior distribution and full revelation impose $\Sigma(0)=\Sigma_0$ and $\Sigma(T)=0$, respectively.

The equations for $A$ and $\Gamma$ in \eqref{eq:A_C_equations} and \eqref{eq:B_E_equations} depend on these functions through the products $2\lambda A$ and $2\lambda\Gamma$, whereas \eqref{eq:filtering_consistency} becomes singular as $\Sigma$ approaches zero. We therefore introduce the dimensionless variables
\begin{equation}
    r(t):=2\lambda(t)A(t),
    \qquad
    s(t):=2\lambda(t)\Gamma(t).
    \label{eq:dimensionless_variables}
\end{equation}
These transformations isolate the factors $1-r$ and $1-s$ and make the terminal singularity explicit. Combining \eqref{eq:A_C_equations}, \eqref{eq:B_E_equations}, \eqref{eq:lambda_D_equations}, and \eqref{eq:variance_equation} yields the closed system
\begin{equation}
    \left\{
    \begin{aligned}
        \Sigma'(t)
        &=
        -\sigma^2\lambda(t)^2,\\
        \lambda'(t)
        &=
        -\sigma^2\lambda(t)^2
        \bigl(
            \rho s(t)^2+\eta r(t)^2
        \bigr),\\
        r'(t)
        &=
        -\sigma^2\lambda(t)
        \bigl(
            \rho s(t)^2+\eta r(t)^2
        \bigr)r(t)
        +
        2\frac{\Sigma'(t)}{\Sigma(t)}
        \bigl(1-r(t)\bigr),\\
        s'(t)
        &=
        2\frac{\Sigma'(t)}{\Sigma(t)}
        \bigl(1-s(t)\bigr)
        +
        \rho\sigma^2\lambda(t)s(t)^2
        \bigl(1-s(t)\bigr)
        -
        \eta\sigma^2\lambda(t)r(t)^2s(t).
    \end{aligned}
    \right.
    \label{eq:coupled_system}
\end{equation}
The full revelation gives $\Sigma(T)=0$, but \eqref{eq:coupled_system} does not determine $\lambda(T-)$, $r(T-)$, or $s(T-)$. We therefore restrict our searching of solutions to the nondegenerate regime: $\lambda(T-)\in(0,\infty)$. Given $\lambda(T-)=b \in(0,\infty)$, \eqref{eq:variance_equation} yields $\Sigma(t)\sim\sigma^2b^2(T-t)$ and $\Sigma'(t)/\Sigma(t)\sim-1/(T-t)$ as $t\uparrow T$. Thus, the equations for $r$ and $s$ in \eqref{eq:coupled_system} contain singular terms of order $(T-t)^{-1}(1-r(t))$ and $(T-t)^{-1}(1-s(t))$, respectively. Consequently, if $r$ and $s$ are required to remain finite and admit finite terminal limits, compatibility with the system forces $r(T-)=s(T-)=1$. Accordingly, we consider the system \eqref{eq:coupled_system} with the terminal condition:
\begin{equation}\label{eq:coupled_system.termial}
\begin{aligned}
\lambda(T-)=b \in(0,\infty),\quad r(T-)=s(T-)=1.
\end{aligned}
\end{equation}
We will show that the terminal value $b\in(0,\infty)$ can be selected uniquely so that $\Sigma(0)=\Sigma_0$ and that the corresponding $\lambda$, $r$, and $s$ remain finite and positive on $[0,T]$. 

Once this system has been solved, $A$ and $\Gamma$ are recovered from \eqref{eq:dimensionless_variables}, while the candidate trading coefficient $k$ follows from \eqref{eq:filtering_consistency}. The functions $C,E$, and $D$ are subsequently determined by \eqref{eq:A_C_equations}--\eqref{eq:lambda_D_equations} and the terminal conditions in \eqref{eq:constant_terminal_conditions}.

% this stage, full revelation determines $\Sigma(T)=0$, but \eqref{eq:coupled_system} does not determine $\lambda(T-)$, $r(T-)$, or $s(T-)$. We therefore restrict attention to the nondegenerate regime $\lambda(T-)=b\in(0,\infty)$ and $r(T-)=s(T-)=1$, thereby excluding a nonpositive or unbounded terminal limit of $\lambda$ and unbounded terminal limits of $r$ and $s$. These conditions are introduced as working hypotheses rather than consequences of \eqref{eq:coupled_system}. Under $\lambda(T-)=b\in(0,\infty)$, \eqref{eq:variance_equation} yields $\Sigma(t)\sim\sigma^2b^2(T-t)$ and $\Sigma'(t)/\Sigma(t)\sim-1/(T-t)$ as $t\uparrow T$. Hence, $r(T-)=s(T-)=1$ are the terminal compatibility conditions that offset the potentially singular terms involving $1-r$ and $1-s$ in the equations for $r$ and $s$. The construction below verifies that these terminal conditions are consistent with the coupled system.

%The terminal value $b$ is initially treated as a shooting parameter. We will show that it can be selected uniquely so that $\Sigma(0)=\Sigma_0$ and that the corresponding $\lambda$, $r$, and $s$ remain finite and positive on $[0,T]$. This establishes the feasibility of the nondegenerate terminal conditions used in the shooting construction.

\noindent\textbf{The risk-neutral case.}
When $\rho=0$, we have $h=m$ and $s=r$. The coupled system \eqref{eq:coupled_system} then reduces to
\begin{equation}
    \left\{
    \begin{aligned}
        \Sigma'(t)
        &=
        -\sigma^2\lambda(t)^2,\\
        \lambda'(t)
        &=
        -\eta\sigma^2\lambda(t)^2r(t)^2,\\
        r'(t)
        &=
        -\eta\sigma^2\lambda(t)r(t)^3
        +
        2\frac{\Sigma'(t)}{\Sigma(t)}
        \bigl(1-r(t)\bigr).
    \end{aligned}
    \right.
    \label{eq:risk_neutral_coupled_system}
\end{equation}
The corresponding nondegenerate terminal conditions are $\lambda(T-)=b\in(0,\infty)$ and $r(T-)=1$. The terminal value $b$ is again treated as a shooting parameter and determined by $\Sigma(0)=\Sigma_0$.

Systems \eqref{eq:coupled_system} and \eqref{eq:risk_neutral_coupled_system} combine the informed trader's intrapersonal equilibrium condition with the market maker's rational pricing condition. Their terminal singularities and the determination of the corresponding shooting parameters are addressed in Section~\ref{sec:proof}.

\subsection{Main theorem}
\label{sec:main_theorem}

We first present the risk-neutral case $\rho=0$ and $\eta>0$, for which the coefficient system admits an explicit solution. We then consider the risk-averse case $\rho,\eta>0$, where a scaling argument reduces the coupled system to a universal initial-value problem independent of the shooting parameter.

\subsubsection{The risk-neutral case}

The following result establishes the unique explicit solution of the risk-neutral coupled system \eqref{eq:risk_neutral_coupled_system}.

\begin{theorem}[Explicit Gaussian--linear solution]
\label{thm:explicit_solution}

Let
\begin{equation}
    \mathcal{T}_0(b)
    :=
    \frac{3}{5\sigma^2\eta b}
    \left[
        1-
        \left(
            1+\frac{\eta\Sigma_0}{3b}
        \right)^{-5}
    \right],\;\; \text{for $b\in(0,\infty)$}.
    \label{eq:coefficient_map}
\end{equation}
% Then $\mathcal{T}_0(b)$ is continuous, strictly decreasing and satisfies 
% $$
% \lim_{b\to 0+} \mathcal T_0 (b)
% $$
% As $b$ increases from zero to infinity, $\mathcal{T}_0(b)$ is continuous and strictly decreasing from infinity to zero. 
Then $\mathcal{T}_0$ is continuous and strictly decreasing on $(0,\infty)$, with $\mathcal{T}_0(b)\to\infty$ as $b\downarrow0$ and $\mathcal{T}_0(b)\to0$ as $b\uparrow\infty$. Hence, there exists a unique $b^*>0$ such that
%\begin{equation}
$    \mathcal{T}_0(b^*)=T.$
%    \label{eq:coefficient_equation}
%\end{equation}

Define $\chi(t):=1-\frac{5\sigma^2\eta b^*}{3}(T-t)$. Then $\chi$ is strictly positive on $[0,T]$, and the risk-neutral coupled system \eqref{eq:risk_neutral_coupled_system} admits, within the nondegenerate terminal class satisfying $\lambda(T-)\in(0,\infty)$, the unique solution
\begin{equation}
    \lambda(t)
    =
    b^*\chi(t)^{-3/5},
    \qquad
    r(t)
    =
    \chi(t)^{-1/5},
    \qquad
    \Sigma(t)
    =
    \frac{3b^*}{\eta}
    \left(
        \chi(t)^{-1/5}-1
    \right).
    \label{eq:explicit_coefficient_solution}
\end{equation}
In particular, $\lambda(T-)=b^*$, $r(T-)=1$, $\Sigma(0)=\Sigma_0$, and $\Sigma(T)=0$. Both $\lambda$ and $r$ remain finite and strictly positive on $[0,T]$.
\end{theorem}

Theorem~\ref{thm:explicit_solution} determines the price impact from the model parameters $(T,\eta,\sigma,\Sigma_0)$. The price impact decreases over time, while the conditional variance decreases to zero as private information is incorporated into the price. Substituting \eqref{eq:explicit_coefficient_solution} back into the equations for $A$ and $k$ yields
\begin{equation}
    A(t)
    =
    \frac{\chi(t)^{2/5}}{2b^*},
    \qquad
    k(t)
    =
    \frac{\sigma^2\eta}{3}
    \frac{\chi(t)^{-3/5}}
    {\chi(t)^{-1/5}-1},
    \qquad 0\leq t<T.
    \label{eq:explicit_A_k}
\end{equation}
Substituting these functions into \eqref{eq:A_C_equations} and \eqref{eq:lambda_D_equations} and using $C(T)=D(T)=0$ yield
\begin{equation}
    C(t)
    =
    \frac{3}{2\eta}
    \left(
        1-\chi(t)^{1/5}
    \right),
    \qquad
    D(t)
    =
    \frac{3}{4\eta}
    \left(
        1-\chi(t)^{2/5}
    \right).
    \label{eq:explicit_C_D}
\end{equation}
Consequently,
\begin{equation}
    m(t,e)
    =
    \frac{\chi(t)^{2/5}}{2b^*}e^2
    +
    \frac{3}{2\eta}
    \left(
        1-\chi(t)^{1/5}
    \right),
    \qquad
    u(t,e)
    =
    \frac{\chi(t)^{3/5}}{2b^*}e^2
    +
    \frac{3}{4\eta}
    \left(
        1-\chi(t)^{2/5}
    \right).
    \label{eq:explicit_m_u}
\end{equation}

The preceding result determines the coefficient functions
$\Sigma$, $\lambda$, $r$, and $k$. We define the associated candidate
equilibrium strategy by
$
\widehat{\alpha}_t^*=k(t)e_t.
$
Before applying the verification argument, we must establish that $\widehat{\alpha}^*$ belongs to the admissible class. 
%The preceding result determines a unique candidate trading strategy $\widehat{\alpha}_t=k(t)e_t$. 
%Before applying the verification argument, we must establish that this strategy belongs to the admissible class. 
This is not immediate because $k(t)$ diverges as $t\uparrow T$. The next proposition shows that the simultaneous decay of the pricing error offsets this singularity.

\begin{proposition}[Admissibility of the candidate strategy]
\label{thm:admissibility}

Let $\widehat{\alpha}^*_s=k(s)e_s$ with $k$ defined in \eqref{eq:explicit_A_k}. For every initial state $(t,e,w)$ with $t<T$, the corresponding state equation is well posed on $[t,T)$, the pricing error satisfies $e_s\to0$ as $s\uparrow T$, and
\begin{equation}
    \int_t^T
    |\widehat{\alpha}^*_s|\,\mathrm{d}s
    <\infty
    \qquad
    \mathbb{P}\text{-a.s.}
    \label{eq:strategy_L1}
\end{equation}
Moreover, for every $p\geq1$,
\begin{equation}
    \mathbb{E}_{t,e,w}
    \left[
        \sup_{t\leq s\leq T}
        |W_s^{\widehat{\alpha}^*}|^p
    \right]
    <\infty.
    \label{eq:wealth_Lp}
\end{equation}
In particular, $W_T^{\widehat{\alpha}^*}\in L^2$ and satisfies
\begin{equation}
    \mathbb{E}_{t,e}
    \left[
        |W_T^{\widehat{\alpha}^*}-w|^2
    \right]
    \leq
    \sigma^4
    \left(
        \int_t^T\lambda(s)\,\mathrm{d}s
    \right)^2
    \left(
        3+\frac{6e^2}{\Sigma(t)}
        +\frac{e^4}{\Sigma(t)^2}
    \right).
    \label{eq:wealth_L2_bound}
\end{equation}
Therefore, $\widehat{\alpha}^*$ belongs to the admissible class introduced in Section~\ref{sec:model}, and admissibility is preserved under the bounded constant pasting
deviations in Definition \ref{def:insider_equilibrium}.
\end{proposition}

Proposition~\ref{thm:admissibility} resolves the apparent conflict between the singular trading intensity and admissibility. As the terminal time approaches, the informed trader trades at an increasing rate, but the informational advantage $e_t=V-P_t$ vanishes simultaneously. The resulting cumulative order and terminal wealth remain well defined.

Having established admissibility of $\widehat{\alpha}^*$, it remains to verify that the candidate obtained from the extended HJB system satisfies the local equilibrium condition. This step is essential because the singular constraint \eqref{eq:singular_constraint} is a first-order condition and does not by itself imply that the candidate is a weak equilibrium.

%We henceforth denote the functions determined in Theorem~\ref{thm:explicit_solution} by $\lambda^*$ and $k^*$. 
Define the corresponding pricing rule and trading strategy by
\begin{equation}
    H^*(t,\xi)
    =
    \mu+\xi,
    \qquad
    P_t^*
    =
    H^*(t,\xi_t^*)
    =
    \mu+\xi_t^*,
    \qquad
    \widehat{\alpha}_t^*
    =
    k(t)(V-P_t^*),
    \label{eq:equilibrium_price_strategy}
\end{equation}
where $\xi^*_t=\int_0^t\lambda(s)\,\mathrm{d}Y_s^*$ and the equilibrium order flow $Y^*$ is generated by $\mathrm{d}Y_t^*=\widehat{\alpha}_t^*\,\mathrm{d}t+\sigma\,\mathrm{d}B_t$.

\begin{theorem}[Linear Gaussian weak equilibrium]
\label{thm:main_equilibrium}

The pricing rule and trading strategy in \eqref{eq:equilibrium_price_strategy} form a weak market equilibrium. The price is rational:
\begin{equation}
    P_t^*
    =
    \mathbb{E}
    \left[
        V\mid\mathcal{F}_t^{Y^*}
    \right],
    \qquad
    0\leq t<T,
    \label{eq:main_rational_pricing}
\end{equation}
and $\operatorname{Var}(V\mid\mathcal{F}_t^{Y^*})=\Sigma(t)$. Moreover, the informed trading is inconspicuous, $\mathbb{E}[\widehat{\alpha}_t^*\mid\mathcal{F}_t^{Y^*}]=0$, and full revelation holds:
\begin{equation}
    P_T^*=V
    \qquad
    \mathbb{P}\text{-a.s.}
    \label{eq:main_full_revelation}
\end{equation}

The functions $g(t,e,w)=w+m(t,e)$ and $U(t,e,w)=w+u(t,e)$, with $m$ and $u$ defined in \eqref{eq:explicit_m_u}, satisfy
\[
    g(t,e,w)
    =
    \mathbb{E}_{t,e,w}
    \left[
        W_T^{\widehat{\alpha}^*}
    \right],
    \qquad
    U(t,e,w)
    =
    \mathbb{E}_{t,e,w}
    \left[
        W_T^{\widehat{\alpha}^*}
    \right]
    -
    \frac{\eta}{2}
    \operatorname{Var}_{t,e,w}
    \left(
        W_T^{\widehat{\alpha}^*}
    \right).
\]
Finally, for every $t<T$, every state $(e,w)$, and every constant $a\in \R$,
\begin{equation}
    \limsup_{\varepsilon\downarrow0}
    \frac{
        J(t,e,w;a\otimes_{t,\eps}\widehat\alpha^* )
        -
        J(t,e,w;\widehat{\alpha}^*)
    }{\varepsilon}
    \leq0.
    \label{eq:main_weak_equilibrium}
\end{equation}
\end{theorem}

Theorem~\ref{thm:main_equilibrium} completes the verification of the candidate. The strategy $\widehat{\alpha}^*$ satisfies the informed trader's intrapersonal equilibrium condition, while the pricing rule satisfies rational pricing for the order flow generated by that strategy. The proofs of Theorem~\ref{thm:explicit_solution}, Proposition~\ref{thm:admissibility}, and Theorem~\ref{thm:main_equilibrium} are deferred to the risk-neutral part of Section~\ref{sec:proof}.

\subsubsection{The risk-averse case}

Throughout this subsubsection, we assume $\rho,\eta>0$. The risk-averse coupled system \eqref{eq:coupled_system} with the terminal condition \eqref{eq:coupled_system.termial} is characterized through the universal initial-value problem on $[0,\infty)$:
%Throughout this subsubsection, we assume $\rho,\eta>0$. The risk-averse coupled system is characterized through the universal initial-value problem
\begin{equation}
    \left\{
    \begin{aligned}
        \widehat{\Lambda}'(y)
        &=
        \rho\widehat{S}(y)^2+\eta\widehat{R}(y)^2,\\
        \widehat{R}'(y)
        &=
        \frac{
            \bigl(
                \rho\widehat{S}(y)^2
                +
                \eta\widehat{R}(y)^2
            \bigr)
            \widehat{R}(y)
        }
        {\widehat{\Lambda}(y)}
        +
        \frac{2\bigl(1-\widehat{R}(y)\bigr)}{y},\\
        \widehat{S}'(y)
        &=
        \frac{2\bigl(1-\widehat{S}(y)\bigr)}{y}
        -
        \frac{
            \rho\widehat{S}(y)^2
            \bigl(1-\widehat{S}(y)\bigr)
        }
        {\widehat{\Lambda}(y)}
        +
        \frac{
            \eta\widehat{R}(y)^2\widehat{S}(y)
        }
        {\widehat{\Lambda}(y)},
    \end{aligned}
    \right.
    \qquad
    \widehat{\Lambda}(0)
    =
    \widehat{R}(0)
    =
    \widehat{S}(0)
    =
    1.
    \label{eq:universal_system}
\end{equation}
Although the equations for $\widehat{R}$ and $\widehat{S}$ appear singular at $y=0$, the result below establishes that \eqref{eq:universal_system} admits a unique global positive solution, which leads to the existence of solution to \eqref{eq:coupled_system}.
%Although the equations for $\widehat{R}$ and $\widehat{S}$ appear singular at $y=0$, the result below establishes that \eqref{eq:universal_system} admits a unique global positive solution. In terms of this solution, define, for $b>0$, the shooting map
%\begin{equation}
%   \mathcal{T}(b)  :=  \frac{1}{\sigma^2b}   \int_0^{\Sigma_0/b}  \frac{1}{\widehat{\Lambda}(y)^2}\,\mathrm{d}y. \label{eq:risk_averse_shooting_map}
%\end{equation}

\begin{theorem}[The risk-averse coefficient system]
\label{thm:risk_averse_coefficient_system}

System \eqref{eq:universal_system} admits a unique global solution $(\widehat{\Lambda},\widehat{R},\widehat{S})\in C^1([0,\infty);(0,\infty)^3)$, where the differential equations hold on $(0,\infty)$ and the derivatives extend continuously to $y=0$. Then, the shooting map defined by
\begin{equation}
\mathcal{T}(b):=\frac{1}{\sigma^2b}\int_0^{\Sigma_0/b}\frac{1}{\widehat{\Lambda}(y)^2}\,\mathrm{d}y.
\label{eq:risk_averse_shooting_map}
\end{equation}
%Moreover, the shooting map $\mathcal{T}$ defined in \eqref{eq:risk_averse_shooting_map} 
is continuous and strictly
decreasing on $(0,\infty)$, with $\mathcal{T}(b)\to\infty$ as
$b\downarrow0$ and $\mathcal{T}(b)\to0$ as $b\uparrow\infty$.
Consequently, there exists a unique $b^*>0$ such that
\begin{equation}
    \mathcal{T}(b^*)=T.
    \label{eq:risk_averse_shooting_equation}
\end{equation}
For this value of $b^*$, let $\Sigma$ be determined implicitly by
\begin{equation}
    T-t
    =
    \frac{1}{\sigma^2(b^*)^2}
    \int_0^{\Sigma(t)}
    \frac{1}
    {\widehat{\Lambda}\bigl(x/b^*\bigr)^2}
    \,\mathrm{d}x,
    \qquad
    0\leq t\leq T,
    \label{eq:risk_averse_variance_representation}
\end{equation}
and define
\begin{equation}
    \lambda(t)
    =
    b^*\widehat{\Lambda}
    \left(
        \frac{\Sigma(t)}{b^*}
    \right),
    \qquad
    r(t)
    =
    \widehat{R}
    \left(
        \frac{\Sigma(t)}{b^*}
    \right),
    \qquad
    s(t)
    =
    \widehat{S}
    \left(
        \frac{\Sigma(t)}{b^*}
    \right).
    \label{eq:risk_averse_coefficient_representation}
\end{equation}
Then $(\Sigma,\lambda,r,s)$ is the unique solution of \eqref{eq:coupled_system} within the nondegenerate terminal class satisfying $\lambda(T-)\in(0,\infty)$. In particular, $\Sigma(0)=\Sigma_0$, $\Sigma(T)=0$, $\lambda(T-)=b^*$, and $r(T-)=s(T-)=1$. The functions $\lambda$, $r$, and $s$ remain finite and strictly positive on $[0,T]$.
\end{theorem}

\begin{remark}
The uniqueness result in Theorem~\ref{thm:risk_averse_coefficient_system} concerns the coefficient system within the Gaussian--linear ansatz and the stated nondegenerate terminal class. It does not assert uniqueness of weak market equilibrium among arbitrary pricing rules and trading strategies.
\end{remark}

Unlike the closed-form solution in Theorem~\ref{thm:explicit_solution}, Theorem~\ref{thm:risk_averse_coefficient_system} characterizes the price impact and conditional variance through the universal functions $(\widehat{\Lambda},\widehat{R},\widehat{S})$. Once the unique shooting parameter has been determined, the remaining coefficients follow directly from the dimensionless variables in \eqref{eq:dimensionless_variables} and the filtering consistency condition \eqref{eq:filtering_consistency}:
\begin{equation}
    A(t)
    =
    \frac{r(t)}{2\lambda(t)},
    \qquad
    \Gamma(t)
    =
    \frac{s(t)}{2\lambda(t)},
    \qquad
    k(t)
    =
    \frac{\sigma^2\lambda(t)}{\Sigma(t)},
    \qquad
    0\leq t<T.
    \label{eq:risk_averse_A_B_k}
\end{equation}
It follows from \eqref{eq:A_C_equations}--\eqref{eq:lambda_D_equations} and $C(T)=E(T)=D(T)=0$ that
\begin{equation}
    C(t)
    =
    \int_t^T
    \sigma^2\lambda(s)^2A(s)\,\mathrm{d}s,
    \qquad
    E(t)
    =
    \int_t^T
    \sigma^2\lambda(s)^2\Gamma(s)\,\mathrm{d}s,
    \qquad
    D(t)
    =
    \frac{1}{2}
    \int_t^T
    \sigma^2\lambda(s)\,\mathrm{d}s.
    \label{eq:risk_averse_C_E_D}
\end{equation}
Accordingly, the continuation functions are %2~define h first
\begin{equation}
    m(t,e)
    =
    A(t)e^2+C(t),
    \qquad
    h(t,e)
    =
    \Gamma(t)e^2+E(t),
    \qquad
    u(t,e)
    =
    \frac{e^2}{2\lambda(t)}+D(t).
    \label{eq:risk_averse_continuation_functions}
\end{equation}
The preceding result determines the coefficient functions
$\Sigma$, $\lambda$, $r$, $s$, and $k$. We define the associated candidate strategy by $\widehat{\alpha}^*_t=k(t)e_t$. Its singular trading intensity requires the same terminal control as in the risk-neutral case. The entropic certainty equivalent introduces the additional requirement that terminal wealth possess the appropriate exponential moment.

\begin{proposition}[Admissibility of the risk-averse candidate]
\label{prop:risk_averse_admissibility}

Let $\widehat{\alpha}^*_s=k(s)e_s$, where $k$ is defined in \eqref{eq:risk_averse_A_B_k}. For every initial state $(t,e,w)$ with $t<T$, the corresponding state equation is well posed on $[t,T)$, the pricing error satisfies $e_s\to0$ as $s\uparrow T$, and
\begin{equation}
    \int_t^T
    |\widehat{\alpha}^*_s|\,\mathrm{d}s
    <\infty
    \qquad
    \mathbb{P}\text{-a.s.}
    \label{eq:risk_averse_absolute_integrability}
\end{equation}
Moreover, for every $p\geq1$,
\begin{equation}
    \mathbb{E}_{t,e,w}
    \left[
        \sup_{t\leq s\leq T}
        |W_s^{\widehat{\alpha}^*}|^p
    \right]
    <\infty,
    \qquad
    \mathbb{E}_{t,e,w}
    \left[
        e^{-\rho W_T^{\widehat{\alpha}^*}}
    \right]
    <\infty.
    \label{eq:risk_averse_wealth_integrability}
\end{equation}
Therefore, $\widehat{\alpha}^*$ belongs to the admissible class introduced in Section~\ref{sec:model}, and admissibility is preserved under the
bounded constant pasting deviations in Definition \ref{def:insider_equilibrium}.
\end{proposition}

The polynomial and exponential estimates in Proposition~\ref{prop:risk_averse_admissibility} ensure that both components of the risk-averse criterion are well defined. It remains to verify that the candidate satisfies the local equilibrium condition rather than merely the first-order constraint derived from the extended HJB system.

%We henceforth denote the functions $\lambda$ and $k$ defined in \eqref{eq:risk_averse_coefficient_representation} and \eqref{eq:risk_averse_A_B_k} by $\lambda^*$ and $k^*$, respectively. 
Define the pricing rule and trading strategy by
\begin{equation}
    H^*(t,\xi)
    =
    \mu+\xi,
    \qquad
    P_t^*
    =
    H^*(t,\xi_t^*)
    =
    \mu+\xi_t^*,
    \qquad
    \widehat{\alpha}_t^*
    =
    k(t)(V-P_t^*),
    \label{eq:risk_averse_equilibrium_price_strategy}
\end{equation}
where $\xi_t^*=\int_0^t\lambda(s)\,\mathrm{d}Y_s^*$ and $\mathrm{d}Y_t^*=\widehat{\alpha}_t^*\,\mathrm{d}t+\sigma\,\mathrm{d}B_t$.

\begin{theorem}[Risk-averse linear Gaussian weak equilibrium]
\label{thm:risk_averse_main_equilibrium}

The pricing rule and trading strategy in \eqref{eq:risk_averse_equilibrium_price_strategy} form a weak market equilibrium. The price is rational:
\begin{equation}
    P_t^*
    =
    \mathbb{E}
    \left[
        V\mid\mathcal{F}_t^{Y^*}
    \right],
    \qquad
    0\leq t<T,
    \label{eq:risk_averse_rational_price}
\end{equation}
and $\operatorname{Var}(V\mid\mathcal{F}_t^{Y^*})=\Sigma(t)$. Moreover, the informed trading is inconspicuous, $\mathbb{E}[\widehat{\alpha}_t^*\mid\mathcal{F}_t^{Y^*}]=0$, and full revelation holds:
\begin{equation}
    P_T^*=V
    \qquad
    \mathbb{P}\text{-a.s.}
    \label{eq:risk_averse_full_revelation}
\end{equation}

The functions $h$, $g(t,e,w)=w+m(t,e)$, and $U(t,e,w)=w+u(t,e)$, with $h$, $m$, and $u$ defined in \eqref{eq:risk_averse_continuation_functions}, satisfy
\begin{equation}
    \left\{
    \begin{aligned}
        w+h(t,e)
        &=
        -\frac{1}{\rho}
        \log
        \mathbb{E}_{t,e,w}
        \left[
            e^{-\rho W_T^{\widehat{\alpha}^*}}
        \right],\\
        g(t,e,w)
        &=
        \mathbb{E}_{t,e,w}
        \left[
            W_T^{\widehat{\alpha}^*}
        \right],\\
        U(t,e,w)
        &=
        -\frac{1}{\rho}
        \log
        \mathbb{E}_{t,e,w}
        \left[
            e^{-\rho W_T^{\widehat{\alpha}^*}}
        \right]
        -
        \frac{\eta}{2}
        \operatorname{Var}_{t,e,w}
        \left(
            W_T^{\widehat{\alpha}^*}
        \right).
    \end{aligned}
    \right.
    \label{eq:risk_averse_value_representation}
\end{equation}
Finally, for every $t<T$ and every state $(e,w)$, and every constant $a\in \R$,
\begin{equation}
    \limsup_{\varepsilon\downarrow0}
    \frac{
        J(t,e,w;a\otimes_{t,\eps}\widehat\alpha^* )
        -
        J(t,e,w;\widehat{\alpha}^*)
    }{\varepsilon}
    \leq0.
    \label{eq:risk_averse_weak_equilibrium}
\end{equation}
\end{theorem}

Theorem~\ref{thm:risk_averse_main_equilibrium} combines the coefficient construction, admissibility estimates, and verification argument. It establishes the informed trader's intrapersonal equilibrium condition and the rationality of the associated pricing rule. The proofs of Theorem~\ref{thm:risk_averse_coefficient_system}, Proposition~\ref{prop:risk_averse_admissibility}, and Theorem~\ref{thm:risk_averse_main_equilibrium} are deferred to the risk-averse part of Section~\ref{sec:proof}.

\section{Financial interpretations and sensitivity analysis}
\label{sec:financial}

This section examines how the preference parameters affect equilibrium price impact, the revelation of private information, and the timing of information-rent extraction. For the numerical illustrations, we normalize $T=\Sigma_0=\sigma=1$. For each parameter pair $(\rho,\eta)$, the shooting parameter $b^*$ is recomputed from \eqref{eq:risk_averse_shooting_equation}, after which $\Sigma$ and $\lambda$ are recovered from \eqref{eq:risk_averse_variance_representation} and \eqref{eq:risk_averse_coefficient_representation}. Thus, all comparisons hold fixed the trading horizon $T$, the prior variance $\Sigma_0$, and the noise-trading volatility $\sigma$ while comparing fully resolved equilibrium paths.

\subsection{Price impact, information revelation, and information rents}

The equilibrium relations
\[
    \widehat{\alpha}_t^*
    =
    k(t)(V-P_t^*),
    \qquad
    k(t)
    =
    \frac{\sigma^2\lambda(t)}{\Sigma(t)},
    \qquad
    \Sigma^{\prime}(t)
    =
    -\sigma^2\lambda(t)^2
\]
describe the connection between informed trading and the market maker's learning process. Trading on the current pricing error $V-P_t^*$ generates current information rents, but aggregate order flow also reveals private information and reduces the pricing errors that can be exploited in the future. Indeed,
\[
    \mathbb{E}
    \left[
        \widehat{\alpha}_t^*(V-P_t^*)
        \mid\mathcal{F}_t^{Y^*}
    \right]
    =
    \sigma^2\lambda(t),
    \qquad
    -\Sigma^{\prime}(t)
    =
    \sigma^2\lambda(t)^2.
\]
Thus, $\lambda(t)$ determines both the expected instantaneous rate of information-rent extraction and the rate at which the market maker learns the private information. Here, the rate of learning is measured by the decline in the conditional variance $\Sigma(t)$.

The dynamic trade-off between exploiting the current pricing error and preserving the future informational advantage is already present in the classical Kyle--Back model. This trade-off should be distinguished from the intrapersonal conflict generated by the mean--variance term. The classical informed trader is dynamically consistent, even though current trading affects future profit opportunities. Under risk neutrality without a mean--variance penalty, the incentives to exploit current information and preserve future information exactly balance, resulting in constant price impact. In the present model,
\[
    \lambda^{\prime}(t)
    =
    -\sigma^2\lambda(t)^2
    \bigl(
        \rho s(t)^2+\eta r(t)^2
    \bigr),
\]
so CARA risk aversion and the mean--variance component alter this balance and generate a declining term structure of price impact whenever at least one preference parameter is positive. Since the market maker remains risk neutral and sets $P_t^*=\mathbb{E}[V\mid\mathcal{F}_t^{Y^*}]$, the variation in $\lambda$ reflects equilibrium inference rather than a risk premium charged by the market maker.

Figure~\ref{fig:term_structure} compares four benchmark specifications: the classical Kyle--Back case $(\rho,\eta)=(0,0)$, the CARA-only case $(\rho,\eta)=(0.5,0)$, the mean--variance-only case $(\rho,\eta)=(0,0.5)$, and the combined-preferences case $(\rho,\eta)=(0.5,0.5)$.

\begin{figure}[htbp]
    \centering
    \includegraphics[width=\textwidth]{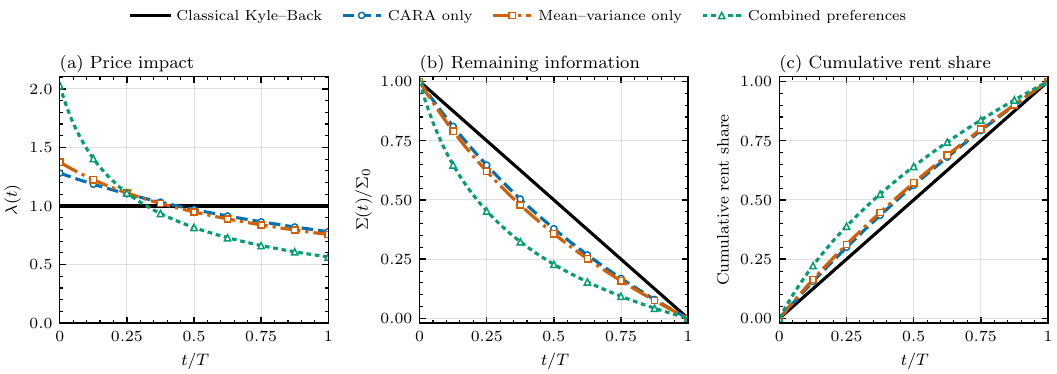}
    \caption{Equilibrium dynamics under alternative preference specifications. The common parameters are $T=\Sigma_0=\sigma=1$. Panel \textnormal{(a)} plots the equilibrium price impact $\lambda(t)$. Panel \textnormal{(b)} plots the remaining-information ratio $\Sigma(t)/\Sigma_0$. Panel \textnormal{(c)} plots the cumulative share of expected information rents extracted by time $t$.}
    \label{fig:term_structure}
\end{figure}

Panel \textnormal{(a)} shows that price impact is constant in the classical benchmark but decreases over the trading horizon in the other three cases. Both preference channels raise price impact near the beginning of the trading horizon and reduce it near the terminal time. This change is most pronounced when $\rho$ and $\eta$ are both positive. Since market depth is measured by $1/\lambda(t)$, CARA risk aversion and the mean--variance penalty make the market shallower early in the trading horizon and deeper later, relative to the classical benchmark.

Panel \textnormal{(b)} translates this term structure into the market maker's learning process. Because $-\Sigma^{\prime}(t)=\sigma^2\lambda(t)^2$, the higher initial price impact under CARA risk aversion and the mean--variance penalty accelerates the early revelation of private information. The combined-preferences case therefore reveals private information earlier than the classical benchmark and either preference channel in isolation. All four paths nevertheless satisfy $\Sigma(T)=0$, so the asset value is fully revealed at the terminal time.

Panel \textnormal{(c)} reports
\[
    \frac{\displaystyle\int_0^t\lambda(s)\,\mathrm{d}s}
    {\displaystyle\int_0^T\lambda(s)\,\mathrm{d}s}.
\]
Since the expected instantaneous information-rent rate is $\sigma^2\lambda(t)$, this ratio represents the fraction of expected information rents extracted by time $t$. The declining price-impact paths imply that information rents are also shifted toward the beginning of the trading horizon. This front-loading is strongest when both preference channels are present. Consequently, the front-loading of price impact in Panel \textnormal{(a)} is accompanied by earlier information revelation and information-rent extraction in Panels \textnormal{(b)} and \textnormal{(c)}.

\subsection{Preference sensitivity and the front-loading of price impact}

The two preference parameters originate from different components of the informed trader's criterion. The parameter $\rho$ measures CARA risk aversion through the conditional certainty equivalent. By contrast, $\eta$ determines the additional penalty on conditional wealth variance and introduces time inconsistency. For $t<s<T$, the law of total variance implies
\begin{equation}
    \operatorname{Var}_t(W_T)
    =
    \mathbb{E}_t
    \left[
        \operatorname{Var}_s(W_T)
    \right]
    +
    \operatorname{Var}_t
    \left(
        \mathbb{E}_s[W_T]
    \right).
    \label{eq:financial_total_variance_decomposition}
\end{equation}
The component $\operatorname{Var}_t(\mathbb{E}_s[W_T])$ enters the conditional variance evaluated at time $t$ because $\mathbb{E}_s[W_T]$ is still random from the perspective of the self acting at that time. By time $s$, however, $\mathbb{E}_s[W_T]$ is $\mathcal{F}_s$-measurable and therefore does not contribute to $\operatorname{Var}_s(W_T)$. The selves acting at times $t$ and $s$ thus assign different conditional variance costs to the same terminal wealth, which creates an incentive to reoptimize at time $s$. The current self must therefore anticipate this future reoptimization when deciding how rapidly to exploit its informational advantage and reveal private information.

Figure~\ref{fig:preference_sensitivity} compares the effects of the two parameters on the term structure of price impact. Panel \textnormal{(a)} varies $\eta$ while fixing $\rho=0.5$, whereas Panel \textnormal{(b)} varies $\rho$ while fixing $\eta=0.5$.

\begin{figure}[htbp]
    \centering
    \includegraphics[width=\textwidth]{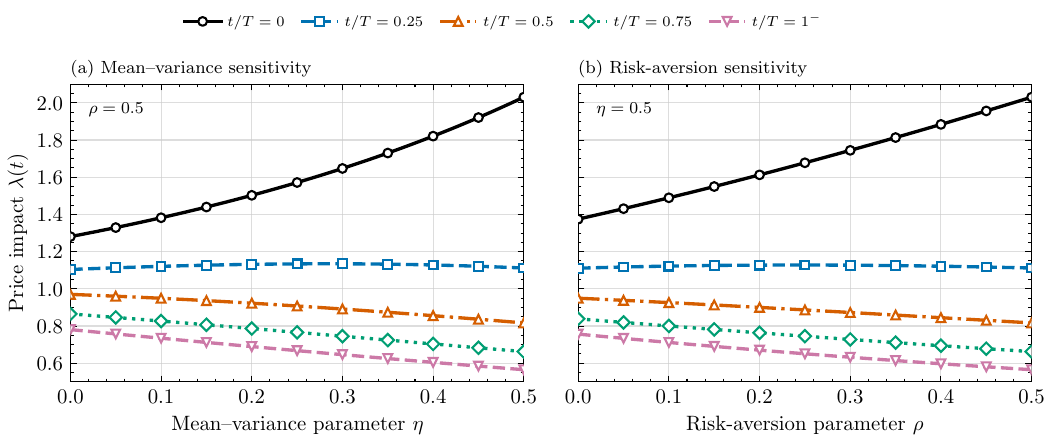}
    \caption{Sensitivity of equilibrium price impact to the preference parameters. The common parameters are $T=\Sigma_0=\sigma=1$. Panel \textnormal{(a)} varies the mean--variance parameter $\eta$ while fixing $\rho=0.5$. Panel \textnormal{(b)} varies the CARA risk-aversion parameter $\rho$ while fixing $\eta=0.5$. The curves report $\lambda(t)$ at $t/T\in\{0,0.25,0.5,0.75,1^{-}\}$, where $1^{-}$ denotes the terminal limit.}
    \label{fig:preference_sensitivity}
\end{figure}

The two panels exhibit a common qualitative pattern. An increase in either preference parameter raises initial price impact but lowers price impact during the later part of the trading horizon. In particular, the terminal value $b^*=\lambda(T-)$ decreases with both parameters. At $t/T=0.25$, the response is substantially weaker and mildly nonmonotone. Hence, neither parameter produces a uniform upward or downward shift of the price-impact curve. Their main effect is to change the temporal profile of price impact.

This pattern reflects the interaction between the differential equation for $\lambda$ and the shooting condition. The terms $\rho s(t)^2$ and $\eta r(t)^2$ both contribute positively to the magnitude of the negative time derivative of $\lambda$. At the same time, the terminal value $b^*$ adjusts endogenously so that the corresponding conditional variance satisfies both $\Sigma(0)=\Sigma_0$ and $\Sigma(T)=0$ over the fixed trading horizon $[0,T]$. In the numerical results, this endogenous adjustment is accompanied by lower terminal price impact and higher initial price impact. Stronger risk sensitivity therefore produces a more steeply declining term structure.

The similarity between the two panels does not imply that $\rho$ and $\eta$ are economically equivalent. The parameter $\rho$ changes the CARA certainty equivalent and operates through $s(t)$, whereas $\eta$ penalizes conditional wealth variance and operates through $r(t)$. Moreover, the CARA-only problem with $\eta=0$ is dynamically consistent, while $\eta>0$ generates the intrapersonal interaction underlying the weak-equilibrium formulation. The common front-loading pattern arises because both channels enter the price-impact dynamics through the positive combination $\rho s(t)^2+\eta r(t)^2$, although the endogenous coefficients $r$ and $s$ generally respond differently.

Through the variance penalty in \eqref{eq:risk_averse_objective} and the decomposition in \eqref{eq:financial_total_variance_decomposition}, a change in $\eta$ affects both the strength of the conditional variance penalty and the associated intrapersonal conflict. The numerical comparison between $\eta=0$ and $\eta>0$ therefore does not isolate the effect of time inconsistency alone. Isolating that effect would require comparison with a precommitment benchmark in which the trader has the same values of $\rho$ and $\eta$ but chooses and commits to the entire future trading strategy at the initial time.

\section{Proof of the main results}
\label{sec:proof}

\subsection{The risk-neutral case}
\label{subsec:risk_neutral_case}

\subsubsection{Solution to the shooting problem}
\label{subsubsec:risk_neutral_shooting_problem}

The risk-neutral coupled system \eqref{eq:risk_neutral_coupled_system} is subject to conditions at opposite ends of the trading interval: the prior variance is $\Sigma(0)=\Sigma_0$, while full revelation imposes $\Sigma(T)=0$. The terminal behavior of $\lambda$ and $r$ is not specified directly by the model. We first seek a solution in the nondegenerate class characterized by $\lambda(T-)=b>0$ and $r(T-)=1$, where $b$ is initially treated as a free parameter. For each $b>0$, we solve the system in the variance variable over $[0,\Sigma_0]$ and determine the resulting length of the trading interval. The shooting problem consists in selecting $b$ so that this length equals the given horizon $T$. Once $b$ has been determined, the resulting explicit solution will be used to verify the proposed terminal behavior.

The equation for $\lambda$ in \eqref{eq:risk_neutral_coupled_system} implies $\lambda'(t)=-\eta\sigma^2\lambda(t)^2r(t)^2\leq0$. Hence, any solution satisfying $\lambda(T-)=b>0$ also satisfies $\lambda(t)\geq b>0$ for $t<T$. It follows from the formula for $\Sigma'(t)$ in \eqref{eq:filtering_equations} that $t\mapsto\Sigma(t)$ is strictly decreasing and therefore admits an inverse. In the following, we may use $x:=\Sigma(t)$ as a backward time variable, with $x=0$ corresponding to $t=T$ and $x=\Sigma_0$ corresponding to $t=0$. The advantage is that the terminal singularity at $\Sigma(T)=0$ is moved to the fixed endpoint $x=0$, and the unknown time horizon $T$ is separated from the system. Write $\Lambda_b(x):=\lambda(t)$ and $R_b(x):=r(t)$. The equations for $\lambda$ and $r$ in \eqref{eq:risk_neutral_coupled_system} become
\begin{equation}
    \left\{
    \begin{aligned}
        \frac{\mathrm{d}\Lambda_b}{\mathrm{d}x}
        &=
        \eta R_b^2,\\
        \frac{\mathrm{d}R_b}{\mathrm{d}x}
        &=
        \frac{\eta R_b^3}{\Lambda_b}
        +
        \frac{2(1-R_b)}{x},
    \end{aligned}
    \right.
    \qquad
    \Lambda_b(0)=b,
    \quad
    R_b(0)=1.
    \label{eq:risk_neutral_variance_time_system}
\end{equation}
To deal with the singularity in second equation of \eqref{eq:risk_neutral_variance_time_system} at $x=0$, we tranform it to an equivalent Volterra formulation
\begin{equation}
    \left\{
    \begin{aligned}
        \Lambda_b(x)
        &=
        b+\eta\int_0^xR_b(s)^2\,\mathrm{d}s,\\
        R_b(x)
        &=
        1+\frac{\eta}{x^2}
        \int_0^x
        \frac{s^2R_b(s)^3}{\Lambda_b(s)}\,\mathrm{d}s.
    \end{aligned}
    \right.
    \label{eq:risk_neutral_volterra_system}
\end{equation}
Indeed, integrating the first equality in \eqref{eq:risk_neutral_variance_time_system} from 0 to $x$ yields the first equality in \eqref{eq:risk_neutral_volterra_system}. Multiplying both sides of the second equality in \eqref{eq:risk_neutral_variance_time_system} by $x^2$, we get
$(x^2 R_b)' = 2x R_b + x^2R_b'= 2x+  \eta x^2 R_b^3/\Lambda_b$,
and hence
$$
x^2 R_b(x) -\lim_{\varepsilon\to 0+} \varepsilon^2 R_b(\varepsilon) = x^2 +\eta \int_0^x \frac{s^2 R_b^3(s) }{\Lambda_b(s)}ds.
$$
Then the second equality in \eqref{eq:risk_neutral_volterra_system} follows from the continuity of $R_b$ at $x=0$.

We first solve the singular variance-time system and establish uniqueness of its positive solution.

\begin{lemma}
\label{lem:risk_neutral_variance_time_solution}
For every $b>0$, system \eqref{eq:risk_neutral_variance_time_system} admits a unique global positive solution of the form
\begin{equation}
    R_b(x)
    =
    1+\frac{\eta x}{3b},
    \qquad
    \Lambda_b(x)
    =
    b\left(1+\frac{\eta x}{3b}\right)^3,
    \qquad x\geq0.
    \label{eq:risk_neutral_variance_time_solution}
\end{equation}
\end{lemma}

\begin{proof}
	Introduce $y:=\eta x/b$, $\Lambda_b(x)=b\widehat{\Lambda}(y)$, and $R_b(x)=\widehat{R}(y)$. By \eqref{eq:risk_neutral_variance_time_system}, the rescaled system is
	\[
	\widehat{\Lambda}'(y)=\widehat{R}(y)^2,
	\qquad
	\widehat{R}'(y)=\frac{\widehat{R}(y)^3}{\widehat{\Lambda}(y)}
	+\frac{2(1-\widehat{R}(y))}{y},
	\qquad
	\widehat{\Lambda}(0)=\widehat{R}(0)=1.
	\]
	The functions $\widehat{R}(y)=1+y/3$ and $\widehat{\Lambda}(y)=(1+y/3)^3$ satisfy this system. Returning to the original variables yields \eqref{eq:risk_neutral_variance_time_solution}, which is finite and strictly positive for every $x\geq0$.

    It remains to prove uniqueness. We remark that the usual ODE uniqueness result cannot be applied here because of the singular initial point $x=0$. Hence in the following analysis, we will first establish local uniqueness through the Volterra formulation, then use standard ODE uniqueness away from $x=0$. For $\delta>0$, let
	\[
	\mathcal{K}_\delta
	:=
	\left\{
	(\Lambda,R)\in C([0,\delta])^2:
	\|\Lambda-b\|_{L^\infty([0,\delta])}\leq\frac{b}{2},
	\quad
	\|R-1\|_{L^\infty([0,\delta])}\leq\frac{1}{2}
	\right\},
	\]
	which is a nonempty closed subset of the Banach space $C([0,\delta])^2$ and is therefore complete under the norm $\|(\Lambda,R)\|_{L^\infty([0,\delta])}	:=	\max\left\{	\|\Lambda\|_{L^\infty([0,\delta])}	\|R\|_{L^\infty([0,\delta])}	\right\}.$
	Define the Volterra operator $\mathcal{V}=(\mathcal{V}_1,\mathcal{V}_2)$ on $\mathcal{K}_\delta$ by
	\[
	\mathcal{V}_1(\Lambda,R)(x)
	:=
	b+\eta\int_0^xR(s)^2\,\mathrm{d}s,
	\qquad
	\mathcal{V}_2(\Lambda,R)(x)
	:=
	1+\frac{\eta}{x^2}
	\int_0^x
	\frac{s^2R(s)^3}{\Lambda(s)}\,\mathrm{d}s,
	\]
	where $\mathcal{V}_2(\Lambda,R)(0):=1$. Since
	\[
	\left|
	\frac{\eta}{x^2}
	\int_0^x
	\frac{s^2R(s)^3}{\Lambda(s)}\,\mathrm{d}s
	\right|
	\leq Cx
	\to0
	\qquad\text{as }x\downarrow0,
	\]
	the function $\mathcal{V}_2(\Lambda,R)$ is continuous at $x=0$.	For $(\Lambda,R)\in\mathcal{K}_\delta$, we have $\Lambda\in[b/2,3b/2]$ and $R\in[1/2,3/2]$. It follows that
	\[
	\|\mathcal{V}_1(\Lambda,R)-b\|_{L^\infty([0,\delta])}
	\leq
	\frac{9\eta\delta}{4},
	\qquad
	\|\mathcal{V}_2(\Lambda,R)-1\|_{L^\infty([0,\delta])}
	\leq
	\frac{9\eta\delta}{4b}.
	\]
	Thus, $\mathcal{V}$ maps $\mathcal{K}_\delta$ into itself whenever
	$0<\delta\leq\frac{2b}{9\eta}$.
	
	The function $(\Lambda,R)\mapsto R^3/\Lambda$ is Lipschitz on
	$[b/2,3b/2]\times[1/2,3/2]$. Hence, for some constant $L>0$ and any
	$(\Lambda_i,R_i)\in\mathcal{K}_\delta$, $i=1,2$,
	\[
	\|\mathcal{V}_1(\Lambda_1,R_1)
	-\mathcal{V}_1(\Lambda_2,R_2)\|_{L^\infty([0,\delta])}
	\leq
	3\eta\delta\|R_1-R_2\|_{L^\infty([0,\delta])},
	\]
	and
	\[
	\|\mathcal{V}_2(\Lambda_1,R_1)
	-\mathcal{V}_2(\Lambda_2,R_2)\|_{L^\infty([0,\delta])}
	\leq
	\frac{\eta L\delta}{3}
	\max\left\{
	\|\Lambda_1-\Lambda_2\|_{L^\infty([0,\delta])},
	\|R_1-R_2\|_{L^\infty([0,\delta])}
	\right\}.
	\]
	Choose $\delta_b>0$ such that
	\[
	0<\delta_b<
	\min\left\{
	\frac{2b}{9\eta},
	\frac{1}{\eta\max\left\{3,\frac{L}{3}\right\}}
	\right\}.
	\]
	Then $\mathcal{V}$ maps $\mathcal{K}_{\delta_b}$ into itself and is a contraction under the uniform norm. The Banach fixed-point theorem yields a unique fixed point of $\mathcal{V}$ in $\mathcal{K}_{\delta_b}$, and hence a positive solution of \eqref{eq:risk_neutral_volterra_system} on $[0,\delta_b]$. Differentiating the Volterra system for $x>0$ shows that this solution satisfies \eqref{eq:risk_neutral_variance_time_system} on $(0,\delta_b]$.
	
	Let $(\widetilde{\Lambda},\widetilde R)$ be any other positive solution with the same initial values. It satisfies \eqref{eq:risk_neutral_volterra_system} and, by continuity, its restriction to $[0,\delta_0]$ belongs to $\mathcal{K}_{\delta_0}$ for some $0<\delta_0\leq\delta_b$. The contraction argument applied on $\mathcal{K}_{\delta_0}$ shows that it coincides with the fixed point on $[0,\delta_0]$.
	
	For $x>0$ and $\Lambda>0$, the right-hand side of \eqref{eq:risk_neutral_variance_time_system} is continuously differentiable, and hence locally Lipschitz, in $(\Lambda,R)$. Standard ODE uniqueness therefore implies that any positive solution with the same initial values coincides with the explicit solution \eqref{eq:risk_neutral_variance_time_solution} throughout their common interval of existence.
	
	It remains only to rule out a finite maximal interval for such a positive solution. Suppose that $(\widetilde{\Lambda},\widetilde R)$ is defined on a maximal interval $[0,x_{\max})$ with $x_{\max}<\infty$. By the preceding uniqueness argument,
	\[
	(\widetilde{\Lambda}(x),\widetilde R(x))=(\Lambda_b(x),R_b(x)),\qquad 0\leq x<x_{\max}.
	\]
	Since the explicit solution \eqref{eq:risk_neutral_variance_time_solution} remains finite and strictly positive at $x_{\max}$, the right-hand side \eqref{eq:risk_neutral_variance_time_system} is regular in a neighborhood of $\bigl(x_{\max},\Lambda_b(x_{\max}),R_b(x_{\max})\bigr).$
	The standard ODE existence theorem therefore extends $(\widetilde{\Lambda},\widetilde R)$ beyond $x_{\max}$, contradicting maximality. Hence $x_{\max}=\infty$, and \eqref{eq:risk_neutral_variance_time_solution} is the unique global positive solution of \eqref{eq:risk_neutral_variance_time_system} on $[0,\infty)$.
\end{proof}

For a fixed $b$, solving the system after the change of variables $x=\Sigma(t)$ does not yet solve the original boundary-value problem, since the resulting solution need not correspond to the given calendar-time horizon $T$. It remains to determine the amount of calendar time required for the conditional variance to decrease from $\Sigma_0$ to $0$.
For a fixed $b>0$, the formula for $\Sigma'$ in \eqref{eq:filtering_equations} with $x=\Sigma(t)$ and $\Lambda_b(x)=\lambda(t)$ implies $\mathrm{d}t=-\mathrm{d}x/(\sigma^2\Lambda_b(x)^2)$. Integrating backward from $\Sigma(T)=0$ to $\Sigma(0)=\Sigma_0$, the length of the trading interval associated with $\Lambda_b$ is
\begin{equation}
    \mathcal{T}_0(b)
    :=
    \int_0^{\Sigma_0}
    \frac{\mathrm{d}x}
    {\sigma^2\Lambda_b(x)^2}.
    \label{eq:risk_neutral_horizon_map}
\end{equation}
Thus, the terminal value $b$ is compatible with the trading horizon precisely when $\mathcal{T}_0(b)=T$.

The next lemma shows that the shooting equation $\mathcal T_0 (b)=T$ selects exactly one terminal price impact.

\begin{lemma}
\label{lem:risk_neutral_unique_terminal_impact}
For every $T>0$, there exists a unique $b^*>0$ such that $\mathcal{T}_0(b^*)=T$.
\end{lemma}

\begin{proof}
Substituting \eqref{eq:risk_neutral_variance_time_solution} into \eqref{eq:risk_neutral_horizon_map} yields
\begin{equation}
    \mathcal{T}_0(b)
    =
    \frac{3}{5\sigma^2\eta b}
    \left[
        1-
        \left(
            1+\frac{\eta\Sigma_0}{3b}
        \right)^{-5}
    \right].
    \label{eq:risk_neutral_horizon_map_explicit}
\end{equation}
This function is continuous on $(0,\infty)$. Set $z:=\eta\Sigma_0/(3b)$ and $f(z):=z[1-(1+z)^{-5}]$ on $(0,\infty)$. Then $\mathcal{T}_0(b)=\frac{9}{5\sigma^2\eta^2\Sigma_0}f(z)$. Since $f'(z)=1-(1+z)^{-5}+5z(1+z)^{-6}>0$, the function $f$ is strictly increasing. Because $z$ is strictly decreasing in $b$, $\mathcal{T}_0$ is strictly decreasing in $b$.

Moreover, \eqref{eq:risk_neutral_variance_time_solution} implies $\Lambda_b(x)\geq b$, and hence by \eqref{eq:risk_neutral_horizon_map}, 
$0<\mathcal{T}_0(b)\leq\Sigma_0/(\sigma^2b^2)$. Therefore,
$\mathcal{T}_0(b)\to0$ as $b\uparrow\infty$. As $b\downarrow0$,
\eqref{eq:risk_neutral_horizon_map_explicit} yields
$\mathcal{T}_0(b)\sim3/(5\sigma^2\eta b)$, so
$\mathcal{T}_0(b)\to\infty$. The intermediate value theorem and strict monotonicity yield the existence and uniqueness of $b^*$.
\end{proof}

\begin{proof}[Proof of Theorem~\ref{thm:explicit_solution}]
Lemma~\ref{lem:risk_neutral_unique_terminal_impact} provides a unique $b^*>0$ satisfying the shooting condition $\mathcal{T}_0(b^*)=T$. For such value $b^*$, the original time variable and the conditional variance are related by
\begin{equation}
    T-t
    =
    \frac{3}{5\sigma^2\eta b^*}
    \left[
        1-
        \left(
            1+\frac{\eta\Sigma(t)}{3b^*}
        \right)^{-5}
    \right].
    \label{eq:risk_neutral_time_variance_relation}
\end{equation}
Recall that $\chi(t)=1-\frac{5\sigma^2\eta b^*}{3}(T-t)$. Solving \eqref{eq:risk_neutral_time_variance_relation} for $\Sigma(t)$, we get
    $\Sigma(t)
    =
    \frac{3b^*}{\eta}
    \left(
        \chi(t)^{-1/5}-1
    \right)$. 
Substitution into \eqref{eq:risk_neutral_variance_time_solution} yields
$\lambda(t)=b^*\chi(t)^{-3/5}$ and $r(t)=\chi(t)^{-1/5}$.

Evaluating the expression for $\Sigma$ at $t=0$ and using $\Sigma(0)=\Sigma_0$ implies $\chi(0)=(1+\eta\Sigma_0/(3b^*))^{-5}>0$. Since $\chi$ is increasing and $\chi(T)=1$, it remains strictly positive on $[0,T]$. Consequently, $\Sigma$, $\lambda$, and $r$ are finite on $[0,T]$, with $\lambda(t)>0$ and $r(t)>0$ throughout the trading interval. Moreover,
\[
    \Sigma(T)=0,
    \quad
    \lambda(T-)=b^*,
    \quad
    r(T-)=1,
\]
%Thus, the solution satisfies the prior condition $\Sigma(0)=\Sigma_0$, the full-revelation condition $\Sigma(T)=0$, and the nondegenerate terminal conditions used in the shooting construction. This verifies the consistency of the terminal specification and 
which completes the proof.
\end{proof}

\subsubsection{Admissibility}
\label{subsubsec:risk_neutral_admissibility}

Before verification, we establish the integrability properties of the candidate strategy and the associated state and wealth processes. Throughout this subsubsection, $\Sigma$, $\lambda$,
and $k$ denote the functions determined in Theorem~\ref{thm:explicit_solution} and
\eqref{eq:explicit_A_k}, and $\widehat{\alpha}_t^*=k(t)e_t.$ Although the feedback coefficient $k(t)$ diverges as $t\uparrow T$, the pricing error vanishes sufficiently rapidly to keep cumulative informed trading and terminal wealth well defined. %{\color{blue}Throughout this subsubsection, we write $\lambda$ and $k$ for $\lambda^*$ and $k^*$, respectively.}

We first derive the representation and moment estimates for the pricing error as $t\uparrow T$.

\begin{lemma}
\label{lem:risk_neutral_pricing_error_moments}
For every initial state $(t,e)$ with $t<T$, the pricing error under the candidate strategy $\widehat{\alpha}^*$ satisfies
\begin{equation}
    e_s
    =
    \frac{\Sigma(s)}{\Sigma(t)}e
    -
    \sigma\Sigma(s)
    \int_t^s
    \frac{\lambda(r)}{\Sigma(r)}\,\mathrm{d}B_r,
    \qquad t\leq s<T.
    \label{eq:risk_neutral_pricing_error_representation}
\end{equation}
Consequently,
\begin{equation}
    \mathbb{E}_{t,e}[e_s^2]
    =
    \Sigma(s)
    -
    \frac{\Sigma(s)^2}{\Sigma(t)}
    +
    \frac{\Sigma(s)^2}{\Sigma(t)^2}e^2.
    \label{eq:risk_neutral_error_second_moment}
\end{equation}
Moreover, $e_s\to0$ almost surely and in $L^p$ for every $p\geq1$ as $s\uparrow T$.
\end{lemma}

\begin{proof}
Under the candidate strategy,
$\mathrm{d}e_s=-\lambda(s)k(s)e_s\,\mathrm{d}s-\sigma\lambda(s)\,\mathrm{d}B_s$.
Recall \eqref{eq:filtering_consistency} and \eqref{eq:variance_equation}, which gives  $k(s)=\sigma^2\lambda(s)/\Sigma(s)$ and
$\Sigma'(s)=-\sigma^2\lambda(s)^2$
%, we have$\lambda(s)k(s)=-\Sigma'(s)/\Sigma(s)$. 
Then applying It\^o's formula to
$e_s/\Sigma(s)$ yields $\mathrm{d}(e_s/\Sigma(s))
=-\sigma\lambda(s)\mathrm{d}B_s/\Sigma(s)$.
Integration from $t$ to $s$ proves \eqref{eq:risk_neutral_pricing_error_representation}.

Taking conditional expectations in \eqref{eq:risk_neutral_pricing_error_representation} and applying It\^o's isometry yield $\mathbb{E}_{t,e}[e_s]=\Sigma(s)e/\Sigma(t)$. Also, by $\frac{\mathrm{d}}{\mathrm{d} r} \frac{1}{\Sigma(r)}= -\frac{\Sigma'(r)}{\Sigma(r)^2}=\frac{\sigma^2 \lambda(r)}{\Sigma(r)^2}$, we get
\[
    \operatorname{Var}_{t,e}(e_s)
    =
    \sigma^2\Sigma(s)^2
    \int_t^s
    \frac{\lambda(r)^2}{\Sigma(r)^2}\,\mathrm{d}r
    =
    \Sigma(s)-\frac{\Sigma(s)^2}{\Sigma(t)}.
\]
Adding the square of the conditional mean proves \eqref{eq:risk_neutral_error_second_moment}.

Set $\widetilde e_s:=e_s/\sqrt{\Sigma(s)}$ and $x_s:=\Sigma(s)/\Sigma(t)$. By \eqref{eq:risk_neutral_pricing_error_representation}, conditional on $e_t=e$, the random variable $e_s$ is Gaussian. Using the expressions for $\mathbb{E}_{t,e}[e_s]$ and $\operatorname{Var}_{t,e}(e_s)$ obtained above, it follows that $\widetilde e_s$ is Gaussian with mean
$\sqrt{x_s}\frac{e}{\sqrt{\Sigma(t)}}$ and variance $1-x_s$. Since $x_s\in[0,1]$, for every $p\geq1$,
\begin{equation}
\sup_{t\leq s<T}	\mathbb{E}_{t,e}	\left[	|\widetilde e_s|^{2p}	\right]
\leq C_p	\left(	1+\frac{|e|^{2p}}{\Sigma(t)^p}\right).
\label{eq:risk_neutral_normalized_error_moments}
\end{equation}
Since $e_s=\sqrt{\Sigma(s)}\widetilde e_s$ and $\lim_{s\to T}\Sigma(s)= 0$, it follows from
\eqref{eq:risk_neutral_normalized_error_moments} that
$$
\mathbb{E}_{t,e}[|e_s|^p]=\Sigma(s)^{p/2}\mathbb{E}_{t,e}[|\widetilde e_s|^p]
\leq\Sigma(s)^{p/2}\left(\mathbb{E}_{t,e}[|\widetilde e_s|^{2p}]	\right)^{1/2}
\to 0, \quad \text{as $s\uparrow T$.}
$$
 Hence, $e_s\to0$ in $L^p$.

For the almost-sure convergence, let
$M_s:=\sigma\int_t^s\lambda(r)/\Sigma(r)\,\mathrm{d}B_r$ and set
$\ell(s):=\langle M\rangle_s=1/\Sigma(s)-1/\Sigma(t)$. Since
$\Sigma(s)\to0$, we have $\ell(s)\to\infty$. By the Dambis--Dubins--Schwarz theorem, $M_s=\widetilde{B}_{\ell(s)}$ for a Brownian motion $\widetilde{B}$. Moreover,
\[
\Sigma(s)M_s=
\frac{\widetilde{B}_{\ell(s)}}{\ell(s)}\frac{\ell(s)}{\ell(s)+1/\Sigma(t)}
\to 0\quad\text{almost surely},
\]
where the convergence follows from the strong law
$\widetilde{B}_\ell/\ell\to0$ as $\ell\to\infty$. This together with \eqref{eq:risk_neutral_pricing_error_representation}
and $\lim_{s\to T}\Sigma(s)e/\Sigma(t)=0$ proves that $e_s\to0$ almost surely.
\end{proof}

The decay of the pricing error now offsets the singularity of $k$, ensuring integrability of cumulative trading and wealth.

\begin{lemma}
\label{lem:risk_neutral_wealth_integrability}
For every initial state $(t,e,w)$ with $t<T$, the candidate strategy
$\widehat{\alpha}_s^*=k(s)e_s$ satisfies
\begin{equation} \label{eq:risk_neutral_absolute_trading_integrability}
    \int_t^T
    |\widehat{\alpha}_s^*|\,\mathrm{d}s
    <\infty
    \qquad\text{almost surely}.
\end{equation}
Moreover, for every $p\geq1$,
\begin{equation} \label{eq:risk_neutral_wealth_maximal_moment}
    \mathbb{E}_{t,e,w}
    \left[
        \sup_{t\leq s\leq T}
        |W_s^{\widehat{\alpha}^*}|^p
    \right]
    <\infty.
\end{equation}
In particular, $W_T^{\widehat{\alpha}^*}\in L^2$, with
\begin{equation}
    \mathbb{E}_{t,e}
    \left[
        \left|
            W_T^{\widehat{\alpha}^*}-w
        \right|^2
    \right]
    \leq
    \sigma^4
    \left(
        \int_t^T\lambda(s)\,\mathrm{d}s
    \right)^2
    \left(
        3+\frac{6e^2}{\Sigma(t)}
        +\frac{e^4}{\Sigma(t)^2}
    \right).
    \label{eq:risk_neutral_terminal_wealth_bound}
\end{equation}
\end{lemma}

\begin{proof}
Since $k(s)=\sigma^2\lambda(s)/\Sigma(s)$ and $e_s=\sqrt{\Sigma(s)}\widetilde e_s$, Tonelli's theorem yields
$$
	\begin{aligned}
		\mathbb{E}_{t,e}	\left[	\int_t^T|\widehat{\alpha}_s^*|\,\mathrm{d}s	\right]
		=	\sigma^2	\int_t^T	\frac{\lambda(s)}{\sqrt{\Sigma(s)}}\mathbb{E}_{t,e}[|\widetilde e_s|]\,\mathrm{d}s
		\leq\sigma^2	\sup_{t\leq s<T}	\mathbb{E}_{t,e}[|\widetilde e_s|]	\int_t^T	\frac{\lambda(s)}{\sqrt{\Sigma(s)}}\,\mathrm{d}s.
	\end{aligned}
$$
By \eqref{eq:risk_neutral_normalized_error_moments} with \(p=1\) and Cauchy--Schwarz inequality, we have 
$$
\sup_{t\le s<T}\mathbb E_{t,e}[|\widetilde e_s|]
\le
\left(
\sup_{t\le s<T}\mathbb E_{t,e}[|\widetilde e_s|^2]
\right)^{1/2}
<\infty.
$$
For the remaining integral, using
$\Sigma'(s)=-\sigma^2\lambda(s)^2$ and
$\lambda(s)\geq\lambda(T-)=b^*>0$, we obtain
\[
    \int_t^T
    \frac{\lambda(s)}{\sqrt{\Sigma(s)}}\,\mathrm{d}s
    =
    \frac{1}{\sigma^2}
    \int_0^{\Sigma(t)}
    \frac{\mathrm{d}x}
    {\lambda(s(x))\sqrt{x}}
    \leq
    \frac{2\sqrt{\Sigma(t)}}{\sigma^2b^*}
    <\infty.
\]
Therefore, $\mathbb{E}_{t,e}[\int_t^T|\widehat{\alpha}_s^*|\,\mathrm{d}s]<\infty$,
which proves \eqref{eq:risk_neutral_absolute_trading_integrability}.

Now combining the wealth dynamics, the candidate strategy \(\widehat\alpha_r^*=k(r)e_r\) and \eqref{eq:filtering_consistency}, we have
\begin{equation}\label{eq:error_est_lm}
W_s^{\widehat\alpha^*}-w
=\int_t^s k(r)e_r^2\,dr
=\sigma^2\int_t^s\lambda(r)\widetilde e_r^2\,dr\ge0,
\end{equation}
It follows by Cauchy--Schwarz inequality that
    $\left(
        W_s^{\widehat{\alpha}^*}-w
    \right)^2
    \leq
    \sigma^4
    \left(
        \int_t^s\lambda(r)\,\mathrm{d}r
    \right)
    \left(
        \int_t^s\lambda(r)\widetilde e_r^4\,\mathrm{d}r
    \right)$. 
By Lemma \ref{lem:risk_neutral_pricing_error_moments}, conditional on \(e_t=e\), \(\widetilde e_s\) is Gaussian with mean
$\sqrt{x_s}e/ \sqrt{\Sigma(t)}$
and variance \(1-x_s\), where \(x_s:=\Sigma(s)/\Sigma(t)\in[0,1]\). Consequently,
$$
\mathbb E_{t,e}[\widetilde e_s^4]
=
x_s^2\frac{e^4}{\Sigma(t)^2}
+6x_s(1-x_s)\frac{e^2}{\Sigma(t)}
+3(1-x_s)^2
\le
3+\frac{6e^2}{\Sigma(t)}
+\frac{e^4}{\Sigma(t)^2}.
$$
Taking expectations and letting $s\uparrow T$ proves
\eqref{eq:risk_neutral_terminal_wealth_bound}.

More generally, the weighted Jensen inequality and
\eqref{eq:risk_neutral_normalized_error_moments} yield
\[
    \mathbb{E}_{t,e}
    \left[
        \left|
            W_T^{\widehat{\alpha}^*}-w
        \right|^p
    \right]
    \leq
    C_p
    \left(
        \int_t^T\lambda(s)\,\mathrm{d}s
    \right)^p
    \left(
        1+\frac{|e|^{2p}}{\Sigma(t)^p}
    \right).
\]
By \eqref{eq:error_est_lm}, $W_s^{\widehat{\alpha}^*}-w$ is nondecreasing,
$\sup_{t\leq s\leq T}|W_s^{\widehat{\alpha}^*}|
\leq |w|+W_T^{\widehat{\alpha}^*}-w$, which proves
\eqref{eq:risk_neutral_wealth_maximal_moment}.
\end{proof}

The preceding estimates allow us to complete the proof of Proposition \ref{thm:admissibility}.
\begin{proof}[Proof of Proposition~\ref{thm:admissibility}]
Lemma \ref{lem:risk_neutral_pricing_error_moments} shows that the	candidate state equation is well posed on $[t,T)$ and that the pricing error satisfies $e_s\to0$ as $s\uparrow T$, both almost surely and in $L^p$ for every $p\geq1$. Lemma \ref{lem:risk_neutral_wealth_integrability} proves that
$$	
\int_t^T|\widehat{\alpha}_s^*|\,\mathrm{d}s<\infty	\;\text{almost surely},\quad \text{and}\;\;  \mathbb{E}_{t,e,w}	\left[	\sup_{t\leq s\leq T}	|W_s^{\widehat{\alpha}^*}|^p	\right]<\infty\;\ \forall p\geq1. 
$$
In particular, $W_T^{\widehat{\alpha}^*}\in L^2$, and \eqref{eq:risk_neutral_terminal_wealth_bound} gives the estimate stated in Proposition~\ref{thm:admissibility}. Since $\widehat{\alpha}^*$ is locally square-integrable on $[t,T)$, it follows that $\widehat{\alpha}^*$ is admissible.
		
It remains to verify that admissibility is preserved under the bounded constant pasting deviations used in Definition \ref{def:insider_equilibrium}. Let $a$ be a bounded constant and consider $a\otimes_{t,\varepsilon}\widehat{\alpha}^*$. For sufficiently small $\varepsilon>0$, the coefficients of the state equation are regular on $[t,t+\varepsilon]$, and hence $(e_{t+\varepsilon}^a,W_{t+\varepsilon}^a)$ has finite moments of every order. Conditional on this state, the strategy agrees with the candidate feedback strategy on $[t+\varepsilon,T)$. Applying Lemmas \ref{lem:risk_neutral_pricing_error_moments} and \ref{lem:risk_neutral_wealth_integrability} from
 $(t+\varepsilon,e_{t+\varepsilon}^a,W_{t+\varepsilon}^a)$ shows that the continuation trading rate is absolutely integrable and that
$
W_T^{a\otimes_{t,\varepsilon}\widehat{\alpha}^*}\in L^p, \quad p\geq1.
$
The strategy is also locally square-integrable on $[t,T)$. Therefore, $a\otimes_{t,\varepsilon}\widehat{\alpha}^*$ is admissible for every sufficiently small $\varepsilon>0$.
\end{proof}

\subsubsection{Proof of the verification theorem}
\label{subsubsec:risk_neutral_verification}

The verification proceeds in three steps. We first identify the candidate continuation functions with the conditional moments of terminal wealth, then verify that the candidate price satisfies the rational pricing condition, and finally establish the weak equilibrium condition under bounded pasting deviations.

Throughout this subsubsection, $\widehat{\alpha}^*(t,e)=k(t)e$, where
$\lambda$ and $k$ are specified in Theorem~\ref{thm:explicit_solution} and \eqref{eq:explicit_A_k}. 
%In the calculations below, we write $\lambda$ and $k$ for the candidate functions $\lambda^*$ and $k^*$, respectively. 
Recall from Subsection~\ref{subsec:gaussian_linear_ansatz}
that the candidate pricing rule is $H^*(t,\xi)=\mu+\xi$. Under the change
of variable $e=v-H^*(t,\xi)=v-\mu-\xi$, the generator in
\eqref{eq:insider_generator} becomes
$\mathcal{L}^a\varphi
=\varphi_t-\lambda(t)a\varphi_e+ea\varphi_w
+\frac{1}{2}\sigma^2\lambda(t)^2\varphi_{ee}$.

Define the candidate continuation functions by
\begin{equation}
    g(t,e,w)=w+m(t,e),
    \qquad
    U(t,e,w)=w+u(t,e),
    \label{eq:risk_neutral_verification_functions}
\end{equation}
where $m$ and $u$ are defined in \eqref{eq:explicit_m_u}. To identify the variance term in $U$, we also need the conditional second moment of terminal wealth. Since
$U
=
g-\frac{\eta}{2}
\left(
\mathbb{E}[W_T^2\mid\mathcal F_t]-g^2
\right)
$, this suggests defining
\begin{equation}
    Q(t,e,w)
    :=
    g(t,e,w)^2
    +
    \frac{2}{\eta}
    \bigl(g(t,e,w)-U(t,e,w)\bigr).
    \label{eq:risk_neutral_second_moment_Q}
\end{equation}
For a fixed $t<T$, let $(t_n)_n$ be an arbitrary deterministic sequence satisfying $t\leq t_n<T$ and $t_n\uparrow T$. The following proposition collects the convergence and uniform-integrability estimates needed to pass to the terminal time.

\begin{proposition}
\label{prop:risk_neutral_verification_uniform_integrability}
The candidate state satisfies $e_{t_n}\to0$ and
$W_{t_n}^{\widehat{\alpha}^*}\to W_T^{\widehat{\alpha}^*}$ in $L^p$
for every $p\geq1$. Moreover,
$\{g(t_n,e_{t_n},W_{t_n})\}_n$ and
$\{Q(t_n,e_{t_n},W_{t_n})\}_n$ are uniformly integrable. For every $n$,
the stochastic integrals in the It\^o decompositions of
$g(s,e_s,W_s)$ and $Q(s,e_s,W_s)$ on $[t,t_n]$ are square-integrable
martingales.
\end{proposition}

\begin{proof}
The convergence of $e_{t_n}$ follows from
Lemma~\ref{lem:risk_neutral_pricing_error_moments}. Since
$W_s^{\widehat{\alpha}^*}-w=\int_t^s k(r)e_r^2\,\mathrm{d}r$ is
nondecreasing, $W_{t_n}^{\widehat{\alpha}^*}\to
W_T^{\widehat{\alpha}^*}$ almost surely. Moreover,
\[
    |W_{t_n}^{\widehat{\alpha}^*}-W_T^{\widehat{\alpha}^*}|^p
    \leq
    (W_T^{\widehat{\alpha}^*}-w)^p.
\]
Lemma~\ref{lem:risk_neutral_wealth_integrability} and dominated convergence therefore
yield convergence in $L^p$.

The explicit forms of $m$ and $u$, together with the boundedness of
their coefficients on $[t,T)$, yield constants $C_1,C_2>0$, independent of $s$, such that
\[
    |g(s,e,w)|+|U(s,e,w)|
    \leq
    C_1(1+|w|+|e|^2),
    \qquad
    |Q(s,e,w)|
    \leq
    C_2(1+|w|^2+|e|^4).
\]
Consequently, for a constant $C_3>0$ independent of $n$,
\[
    |Q(t_n,e_{t_n},W_{t_n})|^2
    \leq
    C_3
    \left(
        1+|W_{t_n}|^4+|e_{t_n}|^8
    \right).
\]
By \eqref{eq:risk_neutral_wealth_maximal_moment},
$\sup_n\mathbb{E}_{t,e,w}[|W_{t_n}|^4]<\infty$. Moreover,
\eqref{eq:risk_neutral_normalized_error_moments} implies
\[
    \sup_n
    \mathbb{E}_{t,e}
    \left[
        |e_{t_n}|^8
    \right]
    \leq
    \Sigma(t)^4
    \sup_{t\leq s<T}
    \mathbb{E}_{t,e}[|\widetilde e_s|^8]
    <\infty.
\]
It follows that
$\sup_n\mathbb{E}_{t,e,w}
[|Q(t_n,e_{t_n},W_{t_n})|^2]<\infty$,
and hence $\{Q(t_n,e_{t_n},W_{t_n})\}_n$ is uniformly integrable.
The same argument, using the lower growth order of $g$, proves the
uniform integrability of $\{g(t_n,e_{t_n},W_{t_n})\}_n$.

It remains to verify the integrability of the stochastic terms. Since
$g_e=m_e$ grows linearly in $e$, the preceding moment estimates imply
\[
    \mathbb{E}_{t,e,w}
    \left[
        \int_t^{t_n}
        \sigma^2\lambda(s)^2
        g_e(s,e_s,W_s)^2\,\mathrm{d}s
    \right]
    <\infty.
\]
Moreover, the explicit form of $Q$ yields $|Q_e(s,e,w)|\leq C_n(1+|w||e|+|e|^3)$ on $[t,t_n]$.
Hölder's inequality, \eqref{eq:risk_neutral_wealth_maximal_moment}, and
\eqref{eq:risk_neutral_normalized_error_moments} imply
\[
    \mathbb{E}_{t,e,w}
    \left[
        \int_t^{t_n}
        \sigma^2\lambda(s)^2
        Q_e(s,e_s,W_s)^2\,\mathrm{d}s
    \right]
    <\infty.
\]
Thus,
\[
    \left\{
        \int_t^r
        \sigma\lambda(s)
        g_e(s,e_s,W_s)\,\mathrm{d}B_s
    \right\}_{t\leq r\leq t_n}
    \quad\text{and}\quad
    \left\{
        \int_t^r
        \sigma\lambda(s)
        Q_e(s,e_s,W_s)\,\mathrm{d}B_s
    \right\}_{t\leq r\leq t_n}
\]
are square-integrable martingales.
\end{proof}

The preceding estimates permit passage to the terminal time and identify the candidate functions with the first two moments of terminal wealth.

\begin{lemma}
\label{lem:risk_neutral_continuation_moments}
The functions $g$ and $U$ defined in \eqref{eq:risk_neutral_verification_functions}
satisfy
\begin{equation}
    g(t,e,w)
    =
    \mathbb{E}_{t,e,w}
    \left[
        W_T^{\widehat{\alpha}^*}
    \right]
    \label{eq:risk_neutral_verification_g}
\end{equation}
and
\begin{equation}
    U(t,e,w)
    =
    \mathbb{E}_{t,e,w}
    \left[
        W_T^{\widehat{\alpha}^*}
    \right]
    -
    \frac{\eta}{2}
    \operatorname{Var}_{t,e,w}
    \left(
        W_T^{\widehat{\alpha}^*}
    \right).
    \label{eq:risk_neutral_verification_U}
\end{equation}
\end{lemma}

\begin{proof}
Under $\widehat{\alpha}^*$, the pricing error and wealth satisfy
$\mathrm{d}e_s=-\lambda(s)k(s)e_s\,\mathrm{d}s
-\sigma\lambda(s)\,\mathrm{d}B_s$ and
$\mathrm{d}W_s=k(s)e_s^2\,\mathrm{d}s$. Since
$\mathcal{L}^{\widehat{\alpha}^*}g=0$, It\^o's formula yields $\mathrm{d}g(s,e_s,W_s)
=-\sigma\lambda(s)g_e(s,e_s,W_s)\,\mathrm{d}B_s$.
Proposition~\ref{prop:risk_neutral_verification_uniform_integrability} shows that
the stochastic integral is a square-integrable martingale on
$[t,t_n]$. Therefore,
\[
    g(t,e,w)
    =
    \mathbb{E}_{t,e,w}
    \left[
        g(t_n,e_{t_n},W_{t_n})
    \right].
\]

The same proposition establishes $e_{t_n}\to0$ and
$W_{t_n}\to W_T^{\widehat{\alpha}^*}$ in $L^p$ for every $p\geq1$.
Since the quadratic coefficient of $m$ remains bounded at the terminal time and
its deterministic term converges to zero,
$m(t_n,e_{t_n})\to0$ in $L^1$. Hence,
$g(t_n,e_{t_n},W_{t_n})\to W_T^{\widehat{\alpha}^*}$ in $L^1$.
Passing to the limit proves \eqref{eq:risk_neutral_verification_g}.

The extended HJB system implies $\mathcal{L}^{\widehat{\alpha}^*}U
-\frac{\eta}{2}\sigma^2\lambda(t)^2g_e^2=0$.
Since $\mathcal{L}^{\widehat{\alpha}^*}g=0$ and
$\mathcal{L}^{\widehat{\alpha}^*}(g^2)
=\sigma^2\lambda(t)^2g_e^2$, the definition of $Q$ implies
$\mathcal{L}^{\widehat{\alpha}^*}Q=0$. It\^o's formula and
Proposition~\ref{prop:risk_neutral_verification_uniform_integrability} therefore
yield
\[
    Q(t,e,w)
    =
    \mathbb{E}_{t,e,w}
    \left[
        Q(t_n,e_{t_n},W_{t_n})
    \right].
\]

The quadratic coefficient of $u$ remains bounded at the terminal time and its
deterministic term converges to zero, so
$u(t_n,e_{t_n})\to0$ in probability. Consequently,
$U(t_n,e_{t_n},W_{t_n})\to W_T^{\widehat{\alpha}^*}$ in probability.
We also have
$g(t_n,e_{t_n},W_{t_n})\to W_T^{\widehat{\alpha}^*}$ in probability.
The continuous mapping theorem yields
$g(t_n,e_{t_n},W_{t_n})^2
\to(W_T^{\widehat{\alpha}^*})^2$ in probability. It follows from
\eqref{eq:risk_neutral_second_moment_Q} that
$Q(t_n,e_{t_n},W_{t_n})
\to(W_T^{\widehat{\alpha}^*})^2$ in probability. The uniform
integrability established in
Proposition~\ref{prop:risk_neutral_verification_uniform_integrability} then yields
\begin{equation}
    Q(t,e,w)
    =
    \mathbb{E}_{t,e,w}
    \left[
        \bigl(W_T^{\widehat{\alpha}^*}\bigr)^2
    \right].
    \label{eq:risk_neutral_verification_Q}
\end{equation}
Substituting \eqref{eq:risk_neutral_verification_g} and \eqref{eq:risk_neutral_verification_Q}
into \eqref{eq:risk_neutral_second_moment_Q} proves \eqref{eq:risk_neutral_verification_U}.
\end{proof}

We next verify that the candidate price coincides with the conditional expectation generated by the candidate order flow.

\begin{lemma}
\label{lem:risk_neutral_rational_candidate_price}
Let $P_t^*=H^*(t,\xi_t^*)=\mu+\xi_t^*$, where
$\xi_t^*=\int_0^t\lambda(s)\,\mathrm{d}Y_s^*$ and
$\mathrm{d}Y_t^*
=\widehat{\alpha}_t^*\,\mathrm{d}t+\sigma\,\mathrm{d}B_t$.
Let $\Sigma$ be the variance function determined in Theorem~\ref{thm:explicit_solution},
%Let $\Sigma^*$ be the candidate variance determined by
\begin{equation}
    \Sigma^{\prime}(t)
    =
    -\frac{k(t)^2}{\sigma^2}
    \Sigma(t)^2,
    \qquad
    \Sigma(0)=\Sigma_0,
    \qquad
    \Sigma(T)=0.
    \label{eq:risk_neutral_candidate_variance_filtering}
\end{equation}
Then
\begin{equation}
    P_t^*
    =
    \mathbb{E}
    \left[
        V\mid\mathcal{F}^{Y^*}_t
    \right],
    \qquad
    \operatorname{Var}
    \left(
        V\mid\mathcal{F}^{Y^*}_t
    \right)
    =
    \Sigma(t),
    \qquad
    0\leq t<T.
    \label{eq:risk_neutral_verified_rational_price}
\end{equation}
Moreover,
$\mathbb{E}[\widehat{\alpha}_t^*\mid\mathcal{F}^{Y^*}_t]=0$ and
$P_T^*=V$ almost surely.
\end{lemma}

\begin{proof}
Fix $\tau<T$. Since $k$ is continuous on $[0,T)$, it is bounded on $[0,\tau]$, so the standard linear Gaussian filtering equations apply on this interval. Under the candidate strategy, the order-flow dynamics satisfy
$\mathrm{d}Y_t^*
=k(t)(V-P_t^*)\,\mathrm{d}t+\sigma\,\mathrm{d}B_t$.
Let
$\overline{P}_t:=\mathbb{E}[V\mid\mathcal{F}^{Y^*}_t]$ and
$\overline{\Sigma}(t)
:=\operatorname{Var}(V\mid\mathcal{F}^{Y^*}_t)$.
Since $P^*$ is adapted to $\mathbb{F}^{Y^*}$, the term
$-k(t)P_t^*\,\mathrm{d}t$ is an observable drift and does not affect
the filtering covariance. Hence, on $[0,\tau]$,
\[
    \overline{\Sigma}'(t)
    =
    -\frac{k(t)^2}{\sigma^2}
    \overline{\Sigma}(t)^2,
    \qquad
    \overline{\Sigma}(0)=\Sigma_0.
\]
Thus, $\overline{\Sigma}$ and $\Sigma$ solve the same initial-value
problem on $[0,\tau]$. Uniqueness of the Riccati equation yields
$\overline{\Sigma}(t)=\Sigma(t)$ for $0\leq t\leq\tau$.

The filtering equation for the conditional mean is
\[
    \mathrm{d}\overline{P}_t
    =
    \frac{k(t)\overline{\Sigma}(t)}{\sigma^2}
    \left[
        \mathrm{d}Y_t^*
        -
        k(t)
        \bigl(\overline{P}_t-P_t^*\bigr)\,\mathrm{d}t
    \right].
\]
Since $\overline{\Sigma}=\Sigma$ and
$k(t)\Sigma(t)/\sigma^2=\lambda(t)$, comparison with
$\mathrm{d}P_t^*=\lambda(t)\,\mathrm{d}Y_t^*$ yields
\[
    \mathrm{d}
    \bigl(\overline{P}_t-P_t^*\bigr)
    =
    -\lambda(t)k(t)
    \bigl(\overline{P}_t-P_t^*\bigr)\,\mathrm{d}t.
\]
Moreover, $\overline{P}_0-P_0^*=0$. The unique solution of this linear
equation is identically zero on $[0,\tau]$. Since $\tau<T$ was arbitrary, this proves \eqref{eq:risk_neutral_verified_rational_price} for every $0\leq t<T$.

The rational pricing condition further implies $\mathbb{E}[\widehat{\alpha}_t^*\mid\mathcal{F}^{Y^*}_t]
=k(t)\mathbb{E}[V-P_t^*\mid\mathcal{F}^{Y^*}_t]=0$.
Finally, Lemma~\ref{lem:risk_neutral_pricing_error_moments} establishes $V-P_t^*\to0$ almost surely as $t\uparrow T$, and hence
$P_T^*=V$ almost surely.
\end{proof}

Having identified the continuation moments and verified rational pricing, we now establish the local equilibrium condition under bounded constant pasting deviations.

\begin{proposition}
\label{prop:risk_neutral_deviation_verification}
For every $t<T$, every state $(e,w)$, and every constant action $a\in \R$,
\begin{equation}
    \lim_{\varepsilon\downarrow0}
    \frac{
        J(t,e,w;a\otimes_{t,\eps}\widehat{\alpha}^* )
        -
        J(t,e,w;\widehat{\alpha}^*)
    }{\varepsilon}
    =0.
    \label{eq:risk_neutral_verified_deviation_condition}
\end{equation}
Consequently, $\widehat{\alpha}^*$ is a weak equilibrium strategy.
\end{proposition}

\begin{proof}
%Let $a\otimes_{t,\eps}\widehat{\alpha}^* $ coincide with $a$ on $[t,t+\varepsilon)$ and with $\widehat{\alpha}^*$ thereafter, 
Denote the state at $t+\varepsilon$ by
$(e_{t+\varepsilon}^a,W_{t+\varepsilon}^a)$. Proposition~\ref{thm:admissibility} ensures that $a\otimes_{t,\eps}\widehat{\alpha}^* $ is admissible for every sufficiently
small $\varepsilon>0$. Since the candidate strategy is followed after
$t+\varepsilon$, Lemma~\ref{lem:risk_neutral_continuation_moments} yields
\[
    \mathbb{E}_{t,e,w}
    \left[
        W_T^{a\otimes_{t,\eps}\widehat{\alpha}^* }
    \right]
    =
    \mathbb{E}_{t,e,w}
    \left[
        g(t+\varepsilon,e_{t+\varepsilon}^a,W_{t+\varepsilon}^a)
    \right]
\]
and
\[
    \mathbb{E}_{t,e,w}
    \left[
        \bigl(W_T^{a\otimes_{t,\eps}\widehat{\alpha}^* }\bigr)^2
    \right]
    =
    \mathbb{E}_{t,e,w}
    \left[
        Q(t+\varepsilon,e_{t+\varepsilon}^a,W_{t+\varepsilon}^a)
    \right].
\]

Let
$\Delta_g(\varepsilon):=
\mathbb{E}_{t,e,w}
[g(t+\varepsilon,e_{t+\varepsilon}^a,W_{t+\varepsilon}^a)]-g(t,e,w)$
and
$\Delta_Q(\varepsilon):=
\mathbb{E}_{t,e,w}
[Q(t+\varepsilon,e_{t+\varepsilon}^a,W_{t+\varepsilon}^a)]-Q(t,e,w)$.
Since $U=g-\frac{\eta}{2}(Q-g^2)$, direct expansion yields
\[
    J(t,e,w;a\otimes_{t,\eps}\widehat{\alpha}^* )-U(t,e,w)
    =
    \bigl(1+\eta g(t,e,w)\bigr)\Delta_g(\varepsilon)
    -
    \frac{\eta}{2}\Delta_Q(\varepsilon)
    +
    \frac{\eta}{2}\Delta_g(\varepsilon)^2.
\]

Fix $\varepsilon_0>0$ such that $t+\varepsilon_0<T$, and let
$(e_s^a,W_s^a)_{t\leq s\leq t+\varepsilon_0}$ denote the state
process generated by the constant action $a$ on this interval. For
every $0<\varepsilon\leq\varepsilon_0$, this process agrees with
the state under $a\otimes_{t,\eps}\widehat{\alpha}^* $ on
$[t,t+\varepsilon]$. Since $\lambda$ is bounded on
$[t,t+\varepsilon_0]$, the explicit representations
\[
    e_s^a
    =
    e-a\int_t^s\lambda(r)\,\mathrm{d}r
    -\sigma\int_t^s\lambda(r)\,\mathrm{d}B_r,
    \qquad
    W_s^a
    =
    w+a\int_t^s e_r^a\,\mathrm{d}r
\]
show that $e_s^a$ and $W_s^a$ are Gaussian, with moments of every
order uniformly bounded for $s\in[t,t+\varepsilon_0]$.

Since $g$, $Q$, and their derivatives have polynomial growth, these
local Gaussian moment bounds imply that the stochastic integrals
in their It\^o decompositions are square-integrable martingales and
that the corresponding generator families are uniformly
integrable. It\^o's formula on $[t,t+\varepsilon]$ therefore
implies
\[
    \begin{aligned}
        \Delta_g(\varepsilon)
        &=
        \mathbb{E}_{t,e,w}
        \left[
            \int_t^{t+\varepsilon}
            \mathcal{L}^ag(s,e_s^a,W_s^a)\,\mathrm{d}s
        \right],\\
        \Delta_Q(\varepsilon)
        &=
        \mathbb{E}_{t,e,w}
        \left[
            \int_t^{t+\varepsilon}
            \mathcal{L}^aQ(s,e_s^a,W_s^a)\,\mathrm{d}s
        \right].
    \end{aligned}
\]
Since $(e_s^a,W_s^a)\to(e,w)$ almost surely as $s\downarrow t$ and
the generators are continuous,
\[
    \mathcal{L}^ag(s,e_s^a,W_s^a)
    \to 
    \mathcal{L}^ag(t,e,w),
    \qquad
    \mathcal{L}^aQ(s,e_s^a,W_s^a)
    \to 
    \mathcal{L}^aQ(t,e,w)
\]
almost surely. The uniform integrability established above upgrades
these convergences to $L^1$. Taking expectations and averaging over
$[t,t+\varepsilon]$ therefore yields
\[
    \frac{\Delta_g(\varepsilon)}{\varepsilon}
    \to 
    \mathcal{L}^ag(t,e,w),
    \qquad
    \frac{\Delta_Q(\varepsilon)}{\varepsilon}
    \to 
    \mathcal{L}^aQ(t,e,w).
\]
In particular, $\Delta_g(\varepsilon)=O(\varepsilon)$ and hence
$\Delta_g(\varepsilon)^2/\varepsilon\to0$. Using \eqref{eq:risk_neutral_second_moment_Q} and the identity
$\mathcal{L}^a(g^2)-2g\mathcal{L}^ag
=\sigma^2\lambda(t)^2g_e^2$, it follows that
\[
    \lim_{\varepsilon\downarrow0}
    \frac{
        J(t,e,w;a\otimes_{t,\eps}\widehat{\alpha}^* )-U(t,e,w)
    }{\varepsilon}
    =
    \mathcal{L}^aU(t,e,w)
    -
    \frac{\eta}{2}
    \sigma^2\lambda(t)^2g_e(t,e,w)^2.
\]
For the candidate functions, the right-hand side equals
\[
    u_t
    +
    \frac{1}{2}\sigma^2\lambda(t)^2u_{ee}
    -
    \frac{\eta}{2}\sigma^2\lambda(t)^2m_e^2
    +
    a\bigl(e-\lambda(t)u_e\bigr).
\]
Since $u$ and $m$ satisfy \eqref{eq:risk_neutral_reduced_extended_hjb}, this
expression is zero for every bounded constant action $a$.
Lemma~\ref{lem:risk_neutral_continuation_moments} also implies $U(t,e,w)=J(t,e,w;\widehat{\alpha}^*)$, which proves
\eqref{eq:risk_neutral_verified_deviation_condition}. The required weak equilibrium
inequality follows immediately.
\end{proof}

\begin{proof}[Proof of Theorem~\ref{thm:main_equilibrium}]
Lemma~\ref{lem:risk_neutral_continuation_moments} identifies $g$ and $U$ with the
continuation mean and mean--variance value generated by
$\widehat{\alpha}^*$. Proposition~\ref{prop:risk_neutral_deviation_verification}
verifies the informed trader's weak equilibrium condition, while
Lemma~\ref{lem:risk_neutral_rational_candidate_price} establishes rational pricing
and full revelation. Therefore,
$\bigl((H^*,\lambda),\widehat{\alpha}^*\bigr)$ is a weak market
equilibrium.
\end{proof}

The preceding verification shows that the Gaussian--linear forms
conjectured in Subsection~\ref{subsec:gaussian_linear_ansatz} are
internally consistent. In particular, the conjectured pricing rule,
feedback trading strategy, and quadratic representations of the
continuation functions jointly satisfy the informed trader's
intrapersonal equilibrium condition and the market maker's rational
pricing condition. Hence, the Gaussian--linear ansatz indeed yields a
weak market equilibrium.

\subsection{The risk-averse case}
\label{subsec:risk_averse_case}

\subsubsection{Solution to the shooting problem}
\label{subsubsec:risk_averse_shooting_problem}

By \eqref{eq:coupled_system},
$\lambda'(t)=-\sigma^2\lambda(t)^2
(\rho s(t)^2+\eta r(t)^2)\leq0$.
Since $\lambda(T-)=b>0$, it follows that
$\lambda(t)\geq b>0$ for $t<T$. Hence, $\Sigma'(t)=-\sigma^2\lambda(t)^2<0$,
and we may use $x:=\Sigma(t)$ as a backward time variable,
with $x=0$ corresponding to $t=T$ and $x=\Sigma_0$
corresponding to $t=0$. Then for an arbitrary $b\in (0,\infty)$, define $\Lambda_b(x):=\lambda(t)$, $R_b(x):=r(t)$, and $S_b(x):=s(t)$. System \eqref{eq:coupled_system}--\eqref{eq:coupled_system.termial} then becomes
\begin{equation}
\left\{
\begin{aligned}
    \Lambda_b'(x)
    &=
    \rho S_b(x)^2+\eta R_b(x)^2,\\
    R_b'(x)
    &=
    \frac{\bigl(\rho S_b(x)^2+\eta R_b(x)^2\bigr)R_b(x)}
    {\Lambda_b(x)}
    +
    \frac{2\bigl(1-R_b(x)\bigr)}{x},\\
    S_b'(x)
    &=
    \frac{2\bigl(1-S_b(x)\bigr)}{x}
    -
    \frac{\rho S_b(x)^2\bigl(1-S_b(x)\bigr)}
    {\Lambda_b(x)}
    +
    \frac{\eta R_b(x)^2S_b(x)}
    {\Lambda_b(x)},
\end{aligned}
\right.
\label{eq:risk_averse_variance_time_system}
\end{equation}
with initial condition
\begin{equation}
    \Lambda_b(0)=b,\qquad R_b(0)=S_b(0)=1.
    \label{eq:risk_averse_variance_time_initial_conditions}
\end{equation}
Although the last two equations in \eqref{eq:risk_averse_variance_time_system} are singular at $x=0$, the following lemma establishes local well-posedness. Unlike the risk-neutral system, the explicit fixed-$b$ solution from Lemma~\ref{lem:risk_neutral_variance_time_solution} is no longer available, so local and global well-posedness must be established separately.

\begin{lemma}
\label{lem:risk_averse_local_variance_time_solution}
For every $b>0$, there exists $\delta_b>0$ such that \eqref{eq:risk_averse_variance_time_system}--\eqref{eq:risk_averse_variance_time_initial_conditions} admits a unique positive solution satisfying
\begin{equation}
    (\Lambda_b,R_b,S_b)\in C^1([0,\delta_b];(0,\infty)^3).
    \label{eq:risk_averse_local_variance_time_solution_class}
\end{equation}
\end{lemma}

\begin{proof}
Fix an arbitrary $b\in(0,\infty)$. The second equality in 	\eqref{eq:risk_averse_variance_time_system} gives
$$
R'_b (x) +  \frac{2R_b(x)}{x} =\frac 2 x + \frac{\bigl(\rho S_b(x)^2+\eta R_b(x)^2\bigr)R_b(x)}   {\Lambda_b(x)}.
$$
Then multiplying both sides by $x^2$ gives
$$
(x^2R_b (x))'=2x+ x^2\frac{\bigl(\rho S_b(x)^2+\eta R_b(x)^2\bigr)R_b(x)}   {\Lambda_b(x)},
$$	
which together with $R_b(0)=1$ implies that 
$$
 R_b(x)=1+\frac{1}{x^2}\int_0^x\frac{y^2\bigl(\rho S_b(y)^2+\eta R_b(y)^2\bigr)R_b(y)}{\Lambda_b(y)}\,\mathrm{d}y,
$$
A similar calculation is applied to $S_b$. In sum, system \eqref{eq:risk_averse_variance_time_system}--\eqref{eq:risk_averse_variance_time_initial_conditions} is equivalent to the Volterra system
\begin{equation}
\left\{
\begin{aligned}
    \Lambda_b(x)
    &=
    b+\int_0^x
    \bigl(\rho S_b(y)^2+\eta R_b(y)^2\bigr)\,\mathrm{d}y,\\
    R_b(x)
    &=
    1+\frac{1}{x^2}
    \int_0^x
    \frac{y^2\bigl(\rho S_b(y)^2+\eta R_b(y)^2\bigr)R_b(y)}
    {\Lambda_b(y)}
    \,\mathrm{d}y,\\
    S_b(x)
    &=
    1+\frac{1}{x^2}
    \int_0^x y^2
    \left[
        -\frac{\rho S_b(y)^2\bigl(1-S_b(y)\bigr)}
        {\Lambda_b(y)}
        +
        \frac{\eta R_b(y)^2S_b(y)}
        {\Lambda_b(y)}
    \right]\mathrm{d}y.
\end{aligned}
\right.
\label{eq:risk_averse_volterra_system}
\end{equation}
Given the above integrands being continuous, the integral terms in the last two equalities are {\color{blue}$O(x^3)$} as $x\to 0+$. Their right-hand sides therefore admit continuous extensions to $x=0$, which are defined to be $1$.

For $z=(z_1,z_2,z_3)\in\mathbb{R}^3$, set $|z|_\infty:=\max\{|z_1|,|z_2|,|z_3|\}$. For $z\in C([0,\delta];\mathbb{R}^3)$, let $\|z\|_{L^\infty([0,\delta])}:=\sup_{0\leq x\leq\delta}|z(x)|_\infty$. For $\delta>0$, define
\[
    \mathcal{K}_\delta
    :=
    \left\{
        (\Lambda,R,S)\in C([0,\delta])^3:
        \|\Lambda-b\|_{L^\infty([0,\delta])} \leq\frac{b}{2},\quad
        \|R-1\|_{L^\infty([0,\delta])}\leq\frac{1}{2},\quad
        \|S-1\|_{L^\infty([0,\delta])}\leq\frac{1}{2}
    \right\}.
\]
The set $\mathcal{K}_\delta$ is a nonempty closed subset of the
Banach space $C([0,\delta];\mathbb{R}^3)$ and is therefore complete. Define the Volterra operator $\mathcal{V}=(\mathcal{V}_1,\mathcal{V}_2,\mathcal{V}_3)$ on $\mathcal{K}_\delta$ by
\[
\begin{aligned}
    \mathcal{V}_1(\Lambda,R,S)(x)
    &:=
    b+\int_0^x
    \bigl(\rho S(y)^2+\eta R(y)^2\bigr)\,\mathrm{d}y,\\
    \mathcal{V}_2(\Lambda,R,S)(x)
    &:=
    1+\frac{1}{x^2}\int_0^x
    \frac{y^2\bigl(\rho S(y)^2+\eta R(y)^2\bigr)R(y)}
    {\Lambda(y)}\,\mathrm{d}y,\\
    \mathcal{V}_3(\Lambda,R,S)(x)
    &:=
    1+\frac{1}{x^2}\int_0^x y^2
    \left[
        -\frac{\rho S(y)^2\bigl(1-S(y)\bigr)}{\Lambda(y)}
        +
        \frac{\eta R(y)^2S(y)}{\Lambda(y)}
    \right]\mathrm{d}y,
\end{aligned}
\]
where $\mathcal{V}_2(\Lambda,R,S)(0)=\mathcal{V}_3(\Lambda,R,S)(0):=1$.
For $(\Lambda,R,S)\in\mathcal{K}_\delta$, we have $\Lambda\in[b/2,3b/2]$ and $R,S\in[1/2,3/2]$. It follows that
\[
    \|\mathcal{V}_1(\Lambda,R,S)-b\|_{L^\infty([0,\delta])}
    \leq
    \frac{9(\rho+\eta)\delta}{4},
    \qquad
    \|\mathcal{V}_j(\Lambda,R,S)-1\|_{L^\infty([0,\delta])}
    \leq
    \frac{9(\rho+\eta)\delta}{4b},
    \quad j=2,3.
\]
Thus, $\mathcal{V}$ maps $\mathcal{K}_\delta$ into itself
whenever $0<\delta\leq2b/[9(\rho+\eta)]$.

For $i=1,2,3$, define $F_i: \R^3 \to \R$ by 
\[
\begin{aligned}
    F_1(\Lambda,R,S)
    &:=
    \rho S^2+\eta R^2,\\
    F_2(\Lambda,R,S)
    &:=
    \frac{(\rho S^2+\eta R^2)R}{\Lambda},\\
    F_3(\Lambda,R,S)
    &:=
    -\frac{\rho S^2(1-S)}{\Lambda}
    +
    \frac{\eta R^2S}{\Lambda}.
\end{aligned}
\]
The functions $F_1,F_2,F_3$ are continuously differentiable on the compact set $K_b:=[b/2,3b/2]\times[1/2,3/2]^2$. Hence, there exists $L_b>0$ depending on $b$ such that
\[
    |F_j(z)-F_j(\widetilde z)|\leq L_b|z-\widetilde z|_\infty,
    \quad \forall z,\widetilde z\in K_b,\; j=1,2,3.
\]
Since $\int_0^x y^2\,\mathrm{d}y=x^3/3$, for any $z,\widetilde z\in\mathcal{K}_\delta$ we have
$$
\begin{aligned}
&\|\mathcal{V}_1(z)-\mathcal{V}_1(\widetilde z)\|_{L^\infty([0,\delta])}
    \leq L_b\delta\|z-\widetilde z\|_{L^\infty([0,\delta])},\\
&\|\mathcal{V}_j(z)-\mathcal{V}_j(\widetilde z)\|_{L^\infty([0,\delta])}
    \leq
    \frac{L_b\delta}{3}\|z-\widetilde z\|_{L^\infty([0,\delta])},
    \quad j=2,3.   
\end{aligned}
$$
Therefore,
\[
    \|\mathcal{V}(z)-\mathcal{V}(\widetilde z)\|_{L^\infty([0,\delta])}
    \leq
    L_b\delta\|z-\widetilde z\|_{L^\infty([0,\delta])}.
\]
Choose $\delta_b>0$ such that
\[
    0<\delta_b<
    \min\left\{
        \frac{2b}{9(\rho+\eta)},
        \frac{1}{L_b}
    \right\}.
\]
Applying the preceding estimates with $\delta=\delta_b$, the operator $\mathcal{V}$ maps $\mathcal{K}_{\delta_b}$ into itself and is a contraction on $\mathcal{K}_{\delta_b}$. The Banach fixed-point theorem therefore yields a unique fixed point of $\mathcal{V}$ in $\mathcal{K}_{\delta_b}$. Since every element of $\mathcal{K}_{\delta_b}$ has strictly positive components, this fixed point is a positive continuous solution of \eqref{eq:risk_averse_volterra_system} on $[0,\delta_b]$. Differentiating \eqref{eq:risk_averse_volterra_system} on $(0,\delta_b]$ shows that it satisfies \eqref{eq:risk_averse_variance_time_system} on this interval.

Any other positive solution with the same initial values belongs to
$\mathcal{K}_\delta$ for sufficiently small $\delta>0$ by
continuity. Hence, it coincides with the fixed point near $x=0$ by
the contraction property, and standard ODE uniqueness for $x>0$
extends this equality to $[0,\delta_b]$.

Finally, continuity of the integrands gives
\[
\Lambda_b(x)=b+(\rho+\eta)x+o(x),\;\;
R_b(x)=1+\frac{\rho+\eta}{3b}x+o(x),\;\;
S_b(x)=1+\frac{\eta}{3b}x+o(x),\; \;\text{as $x\downarrow0$.}
\]
 Substituting these expansions into
\eqref{eq:risk_averse_variance_time_system} shows that the
derivatives extend continuously to $x=0$, which proves
\eqref{eq:risk_averse_local_variance_time_solution_class}.
\end{proof}

We next show that the local solution extends uniquely for all $x\geq0$.

\begin{lemma}
\label{lem:risk_averse_global_variance_time_solution}
For every $b>0$, the local solution in \eqref{eq:risk_averse_local_variance_time_solution_class} extends uniquely to a positive solution of \eqref{eq:risk_averse_variance_time_system} on $[0,\infty)$.
\end{lemma}

\begin{proof}
	The scaling $\Lambda_b(x)=b\widehat{\Lambda}(x/b)$, $R_b(x)=\widehat{R}(x/b)$, and $S_b(x)=\widehat{S}(x/b)$ reduces \eqref{eq:risk_averse_variance_time_system} to the universal system \eqref{eq:universal_system}. It therefore suffices to establish global existence for $(\widehat{\Lambda},\widehat{R},\widehat{S})$ on $y\in(0,\infty)$.
	
 To obtain quantities that are easier to control, it is convenient to normalize the variables by $\widehat{\Lambda}$. 
Let $\beta:=1/\widehat{\Lambda}$, $\upsilon:=\widehat{R}/\widehat{\Lambda}$, and $\omega:=\widehat{S}/\widehat{\Lambda}$. A direct calculation yields
	\begin{equation}
		\beta'=-\rho\omega^2-\eta\upsilon^2,\qquad
		\upsilon'=\frac{2(\beta-\upsilon)}{y},\qquad
		\omega'=\frac{2(\beta-\omega)}{y}-\rho\omega^2.
		\label{eq:risk_averse_transformed_universal_system}
	\end{equation}
	To remove the singularity at $y=0$ in \eqref{eq:risk_averse_transformed_universal_system}, let $y=e^z$ and define $\pi(z):=y\beta(y)$, $\zeta(z):=y\upsilon(y)$, and $\psi(z):=y\omega(y)$. It follows from \eqref{eq:risk_averse_transformed_universal_system} that
	\begin{equation}
		\left\{
		\begin{aligned}
			\frac{d\pi}{dz}&=\pi-\rho\psi^2-\eta\zeta^2,\\
			\frac{d\zeta}{dz}&=2\pi-\zeta,\\
			\frac{d \psi}{dz}&=2\pi-\psi-\rho\psi^2.
		\end{aligned}
		\right.
		\label{eq:risk_averse_autonomous_shooting_system}
	\end{equation}
	
	Now we analyze the system \eqref{eq:risk_averse_autonomous_shooting_system}. For $b=1$, Lemma~\ref{lem:risk_averse_local_variance_time_solution} gives the existence of a strictly positive solution $(\widehat{\Lambda},\widehat{R},\widehat{S})$ to the system \eqref{eq:universal_system} on $[0,\delta_1]$, which yields a strictly positive solution to \eqref{eq:risk_averse_autonomous_shooting_system} on $(-\infty,\tilde\delta_1]$, where $\tilde\delta_1:=\log\delta_1$. Let $(-\infty,z_{\max})$, with $z_{\max}\in(\tilde\delta_1,\infty]$, be its maximal interval of existence. Since $\widehat{\Lambda}(y),\widehat{R}(y),\widehat{S}(y)\to1$ as $y\downarrow0$, we have
	\[
	(\pi(z),\zeta(z),\psi(z))\to(0,0,0)
	\qquad\text{as }z\to-\infty.
	\]
	By differentiating $e^z(\psi-\pi)$ and $e^z(\zeta-\psi)$, respectively, and using \eqref{eq:risk_averse_autonomous_shooting_system}, then integrating over $(-\infty,z]$, we obtain
	$$
	\begin{aligned}
		\psi(z)-\pi(z)
		&=\eta\int_{-\infty}^{z}e^{-(z-s)}\zeta(s)^2\,\mathrm{d}s\geq0,\\
		\zeta(z)-\psi(z)
		&=\rho\int_{-\infty}^{z}e^{-(z-s)}\psi(s)^2\,\mathrm{d}s\geq0.
	\end{aligned}
	$$
	Hence,
$
\pi\leq\psi\leq\zeta\text{ on }(-\infty,z_{\max}).
$

The strict positivity of $\pi$ is essential for recovering the original variables from the transformed system. To prove that it is preserved throughout $(-\infty,z_{\max})$, consider
	\[
	\mathcal{H}(\pi,\zeta,\psi)
	=
	\pi^2-\zeta\pi+\frac{\eta}{3}\zeta^3+\frac{\rho}{3}\psi^3
	+\frac{1}{2}(\zeta-\psi)^2.
	\]
	Along solutions of \eqref{eq:risk_averse_autonomous_shooting_system}, we have
	$
	\frac{d\mathcal{H}}{dz}
	=
	-\bigl(\zeta-\psi-\rho\psi^2\bigr)^2
	\leq0$. 
	Since $\mathcal{H}\to0$ as $z\to-\infty$, it follows that $\mathcal{H}\leq0$. Suppose that $\pi$ first reaches zero at a finite time $z_0<z_{\max}$. Since $\pi>0$ on $(-\infty,z_0)$, the second equation in \eqref{eq:risk_averse_autonomous_shooting_system} gives $\zeta(z_0)=2\int_{-\infty}^{z_0}e^{-(z_0-s)}\pi(s)\,\mathrm{d}s>0$.
	Therefore,
	\[
	\mathcal{H}\bigl(0,\zeta(z_0),\psi(z_0)\bigr)
	=
	\frac{\eta}{3}\zeta(z_0)^3
	+
	\frac{\rho}{3}\psi(z_0)^3
	+
	\frac{1}{2}\bigl(\zeta(z_0)-\psi(z_0)\bigr)^2
	>0,
	\]
	which contradicts $\mathcal{H}\leq0$. Thus, $\pi$ remains strictly positive on $(-\infty,z_{\max})$.
	
	Moreover, since $0<\pi\leq\psi\leq\zeta$, $
	\frac{d\pi}{dz}
	=
	\pi-\rho\psi^2-\eta\zeta^2
	\leq
	\pi-(\rho+\eta)\pi^2.
$
	Since $\pi(z)\to0$ as $z\to-\infty$, we conclude that $\pi(z)\leq\frac{1}{\rho+\eta}.$
	Indeed, $\pi$ is below $1/(\rho+\eta)$ for sufficiently negative $z$, and at any point where $\pi=1/(\rho+\eta)$ the preceding inequality gives $d\pi/dz\leq0$, so $\pi$ cannot cross this level from below. Furthermore, we have
$
	\zeta(z)
	=
	2\int_{-\infty}^{z}e^{-(z-s)}\pi(s)\,\mathrm{d}s
	\leq
	\frac{2}{\rho+\eta}.
$
	Therefore,
	\[
	0<\pi(z)\leq\psi(z)\leq\zeta(z)\leq\frac{2}{\rho+\eta},
	\qquad z<z_{\max}.
	\]
	The vector field in \eqref{eq:risk_averse_autonomous_shooting_system} is polynomial and hence locally Lipschitz on $\mathbb R^3$. If $z_{\max}<\infty$, the preceding estimate shows that $(\pi,\zeta,\psi)$ remains bounded on $(-\infty,z_{\max})$. The standard continuation theorem would therefore extend the solution beyond $z_{\max}$, contradicting its maximality. Hence, $z_{\max}=\infty.$
	
	Since $\pi(z)>0$ for every $z\in\mathbb R$, the inverse transformation
	\[
	\widehat{\Lambda}(y)=\frac{y}{\pi(\log y)},\qquad
	\widehat{R}(y)=\frac{\zeta(\log y)}{\pi(\log y)},\qquad
	\widehat{S}(y)=\frac{\psi(\log y)}{\pi(\log y)}
	\]
	is well defined for every $y>0$. On every compact interval $[a,M]\subset(0,\infty)$, continuity and strict positivity of $\pi$ imply that $\pi$ is bounded away from zero, and hence $\widehat{\Lambda}$, $\widehat{R}$, and $\widehat{S}$ remain finite and strictly positive. Together with Lemma~\ref{lem:risk_averse_local_variance_time_solution} near $y=0$, this proves that $(\widehat{\Lambda},\widehat{R},\widehat{S})$ exists uniquely and remains finite and strictly positive on $[0,\infty)$. The scaling then gives the unique global positive solution $(\Lambda_b,R_b,S_b)$ of \eqref{eq:risk_averse_variance_time_system} for every $b>0$.
\end{proof}

The global solution permits us to study the dependence of the constructed time horizon on the terminal value $b$.

\begin{lemma}
\label{lem:risk_averse_shooting_map}
The shooting map $\mathcal{T}$ defined in \eqref{eq:risk_averse_shooting_map} is continuous and strictly decreasing on $(0,\infty)$, and satisfies
\[
  \lim_{b\downarrow0}\mathcal{T}(b)=\infty,\quad
 \lim_{b\uparrow\infty}\mathcal{T}(b)=0.
\]
Moreover, the shooting equation \eqref{eq:risk_averse_shooting_equation} admits a unique solution $b^*>0$.
\end{lemma}

\begin{proof}
Lemma~\ref{lem:risk_averse_global_variance_time_solution} implies $\mathcal{T}$ defined in \eqref{eq:risk_averse_shooting_map} is well-defined.
Set $L:=\Sigma_0/b$ and $F(L):=L\int_0^L\widehat{\Lambda}(y)^{-2}\,\mathrm{d}y$. Then
$
    \mathcal{T}(b)
    =
    \frac{1}{\sigma^2\Sigma_0}F(\Sigma_0/b).
$
Since $\widehat{\Lambda}$ is continuous and strictly positive,
\[
    F'(L)
    =
    \int_0^L\frac{1}{\widehat{\Lambda}(y)^2}\,\mathrm{d}y
    +
    \frac{L}{\widehat{\Lambda}(L)^2}
    >0.
\]
Thus, $F$ is continuous and strictly increasing, and $\mathcal{T}$ is continuous and strictly decreasing.

By $\widehat{\Lambda}(0)=1$, we have $F(L)\sim L^2$ as $L\downarrow0$, and hence $\mathcal{T}(b)\sim \frac{\Sigma_0}{\sigma^2 b^2 }\to0$ as $b\uparrow\infty$. Moreover,
\[
    F(L)
    \geq
    L\int_0^1\frac{1}{\widehat{\Lambda}(y)^2}\,\mathrm{d}y
    \to \infty,\quad \text{as $L\to\infty$.}
\]
Therefore, $\mathcal{T}(b)\to\infty$ as $b\downarrow0$. Then the existence and uniqueness of $b^*$ follows.
\end{proof}

\begin{proof}[Proof of Theorem~\ref{thm:risk_averse_coefficient_system}]
Let $b^*>0$ be the unique solution of \eqref{eq:risk_averse_shooting_equation}, and let $\Sigma$, $\lambda$, $r$, and $s$ be defined by \eqref{eq:risk_averse_variance_representation} and \eqref{eq:risk_averse_coefficient_representation}.
Differentiating \eqref{eq:risk_averse_variance_representation} yields $\Sigma'(t)=-\sigma^2\lambda(t)^2$. Together with \eqref{eq:universal_system}, this shows that $(\Sigma,\lambda,r,s)$ satisfies \eqref{eq:coupled_system}. Moreover,
\[
    \Sigma(0)=\Sigma_0,\qquad
    \Sigma(T)=0,\qquad
    \lambda(T-)=b^*,\qquad
    r(T-)=s(T-)=1.
\]

It remains to prove uniqueness. Let
$(\widetilde{\Sigma},\widetilde{\lambda},
\widetilde r,\widetilde s)$
be another solution of \eqref{eq:coupled_system} satisfying
$\widetilde{\Sigma}(0)=\Sigma_0$,
$\widetilde{\Sigma}(T)=0$,
$\widetilde{\lambda}(T-)\in(0,\infty)$, and
$\widetilde r(T-)=\widetilde s(T-)=1$, and set
$\widetilde b:=\widetilde{\lambda}(T-)$.
The change of variables $x=\widetilde{\Sigma}(t)$ and
Lemmas~\ref{lem:risk_averse_local_variance_time_solution} and
\ref{lem:risk_averse_global_variance_time_solution} imply that the transformed
solution is uniquely determined by $\widetilde b$. Moreover,
\[
    T-t
    =
    \int_0^{\widetilde{\Sigma}(t)}
    \frac{1}{\sigma^2\Lambda_{\widetilde b}(x)^2}
    \,\mathrm{d}x.
\]
The condition $\widetilde{\Sigma}(0)=\Sigma_0$ therefore implies
$\mathcal{T}(\widetilde b)=T$. Lemma~\ref{lem:risk_averse_shooting_map} yields
$\widetilde b=b^*$. The uniqueness of the variance-time solution and
of the preceding time transformation then implies
$(\widetilde{\Sigma},\widetilde{\lambda},\widetilde r,\widetilde s)
=(\Sigma,\lambda,r,s)$.
\end{proof}

\subsubsection{Admissibility}
\label{subsubsec:risk_averse_admissibility}

We next establish admissibility of the risk-averse candidate $\widehat\alpha^*$ associated with the coefficient functions constructed in Subsection \ref{subsubsec:risk_averse_shooting_problem}. Throughout this subsubsection, $\lambda$, $k$, and $\Sigma$ denote the risk-averse candidate coefficients specified in Theorem~\ref{thm:risk_averse_coefficient_system} and \eqref{eq:risk_averse_A_B_k}. The representation and polynomial moment estimates for the pricing error and wealth depend only on the filtering consistency relation, the variance equation, and the nondegenerate terminal behavior of the price impact. They can therefore be inherited from the risk-neutral analysis. The additional issue when $\rho>0$ is the exponential integrability required by the entropic certainty equivalent.

\begin{lemma}
\label{lem:risk_averse_candidate_moments}
For every initial state $(t,e,w)$ with $t<T$, the state equation under
$\widehat{\alpha}_s^*=k(s)e_s$ is well posed on $[t,T)$. With the
risk-averse coefficients substituted for $\lambda$ and $\Sigma$, the
pricing-error representation
\eqref{eq:risk_neutral_pricing_error_representation} and the
second-moment identity
\eqref{eq:risk_neutral_error_second_moment} remain valid.

Moreover, define, for every $p\geq1$ that,
$
\widetilde e_s:=\frac{e_s}{\sqrt{\Sigma(s)}}
$ on $s\in[t, T)$,
then there exists a constant $C_p>0$ such that
\begin{equation}
    \sup_{t\leq s<T}
    \mathbb{E}_{t,e}
    \left[
        |\widetilde e_s|^{2p}
    \right]
    \leq
    C_p
    \left(
        1+\frac{|e|^{2p}}{\Sigma(t)^p}
    \right).
    \label{eq:risk_averse_normalized_error_moments}
\end{equation}
In particular, $e_s\to0$ almost surely and in $L^p$ for every
$p\geq1$ as $s\uparrow T$.

The candidate strategy $\widehat{\alpha}^*$ is locally square-integrable
on $[t,T)$, and satisfies
\begin{equation}
    \int_t^T
    |\widehat{\alpha}_s^*|\,\mathrm{d}s
    <\infty
    \quad
    \mathbb{P}\text{-a.s.},
    \label{eq:risk_averse_proof_absolute_integrability}
\end{equation}
and, for every $p\geq1$,
\begin{equation}
    \mathbb{E}_{t,e,w}
    \left[
        \sup_{t\leq s\leq T}
        |W_s^{\widehat{\alpha}^*}|^p
    \right]
    <\infty.
    \label{eq:risk_averse_proof_wealth_moments}
\end{equation}
Consequently, $W_T^{\widehat{\alpha}^*}\in L^2$.
\end{lemma}

\begin{proof}
Under the risk-averse candidate, the structural identities
\[
    k(s)
    =
    \frac{\sigma^2\lambda(s)}{\Sigma(s)},
    \qquad
    \Sigma'(s)
    =
    -\sigma^2\lambda(s)^2
\]
are unchanged. Theorem~\ref{thm:risk_averse_coefficient_system}
further implies that $\lambda$ remains finite and strictly positive
on $[0,T]$, with $\lambda(T-)=b^*>0$. Hence, the proof of
Lemma~\ref{lem:risk_neutral_pricing_error_moments} applies without
modification and yields the pricing-error representation, the
second-moment identity, and
\eqref{eq:risk_averse_normalized_error_moments}. In particular,
$e_s\to0$ almost surely and in $L^p$ for every $p\geq1$ as
$s\uparrow T$.

The proof of Lemma~\ref{lem:risk_neutral_wealth_integrability} also
depends only on the preceding structural identities and the boundedness
of $\lambda$. It therefore remains valid for the risk-averse
coefficients and yields
\eqref{eq:risk_averse_proof_absolute_integrability} and
\eqref{eq:risk_averse_proof_wealth_moments}. Local square-integrability
follows from the continuity of the state process and the boundedness of
$k$ on every compact subinterval of $[t,T)$.
\end{proof}

The next lemma verifies the additional exponential-integrability condition and records a bound that will also be used in the verification of the continuation certainty equivalent.

\begin{lemma}
\label{lem:risk_averse_exponential_integrability}
For every initial state $(t,e,w)$ with $t<T$ and every $\theta>0$,
\begin{equation}
    \mathbb{E}_{t,e,w}
    \left[
        e^{-\theta W_T^{\widehat{\alpha}^*}}
    \right]
    \leq e^{-\theta w}
    <\infty.
    \label{eq:risk_averse_exponential_moment_bound}
\end{equation}
Moreover, along the candidate state process,
\begin{equation}
    0<
    e^{-\rho(W_s^{\widehat{\alpha}^*}+h(s,e_s))}
    \leq e^{-\rho w},
    \qquad t\leq s<T.
    \label{eq:risk_averse_entropic_process_bound}
\end{equation}
\end{lemma}

\begin{proof}
Since $k(s)>0$, the wealth process under the candidate satisfies
\[
    W_s^{\widehat{\alpha}^*}-w
    =
    \int_t^s k(r)e_r^2\,\mathrm{d}r
    \geq0.
\]
Thus, $W_T^{\widehat{\alpha}^*}\geq w$, and \eqref{eq:risk_averse_exponential_moment_bound} follows immediately. In addition, \eqref{eq:risk_averse_A_B_k} and \eqref{eq:risk_averse_C_E_D} imply $\Gamma(s)>0$ and $E(s)\geq0$. Hence, $h(s,e_s)=\Gamma(s)e_s^2+E(s)\geq0$, which proves \eqref{eq:risk_averse_entropic_process_bound}.
\end{proof}

The preceding estimates allow us to complete the proof of
Proposition~\ref{prop:risk_averse_admissibility}.

\begin{proof}[Proof of Proposition~\ref{prop:risk_averse_admissibility}]

Lemmas~\ref{lem:risk_averse_candidate_moments} and \ref{lem:risk_averse_exponential_integrability} establish all the admissibility
	requirements for $\widehat{\alpha}^*$. In particular, the candidate
	strategy is locally square-integrable on $[t,T)$, and
	$\int_t^T|\widehat{\alpha}_s^*|\,\mathrm{d}s<\infty$ almost surely,
	its terminal wealth has moments of every polynomial order, and
	\[
	\mathbb E_{t,e,w}\left[e^{-\rho W_T^{\widehat{\alpha}^*}}\right]<\infty.
	\]
	Hence, $\widehat{\alpha}^*$ is admissible.
	
	We only need to verify that admissibility is preserved under the bounded
	constant pasting deviations used in Definition~\ref{def:insider_equilibrium}. Let $a$ be a
	bounded constant and consider
	$a\otimes_{t,\varepsilon}\widehat{\alpha}^*$.
	For sufficiently small $\varepsilon>0$, the coefficients of the state
	equation are regular on $[t,t+\varepsilon]$, and hence
	$(e_{t+\varepsilon}^a,W_{t+\varepsilon}^a)$ has finite moments of every
	order. Conditional on this state, the candidate feedback strategy is
	followed on $[t+\varepsilon,T)$, so Lemma~\ref{lem:risk_averse_candidate_moments} gives the required
	polynomial integrability.

\begin{comment}
\begin{proposition}
\label{prop:risk_averse_deviation_admissibility}
The candidate feedback strategy $\widehat{\alpha}^*$ is admissible. Moreover, if $a$ is a bounded constant and $a\otimes_{t,\eps}\widehat\alpha^*  $ coincides with $a$ on $[t,t+\varepsilon)$ and with $\widehat{\alpha}^*$ thereafter, then $a\otimes_{t,\eps}\widehat\alpha^*  $ is admissible for every sufficiently small $\varepsilon>0$.
\end{proposition}

Lemmas~\ref{lem:risk_averse_candidate_moments} and \ref{lem:risk_averse_exponential_integrability} establish all the admissibility requirements for $\widehat{\alpha}^*$.

On $[t,t+\varepsilon)$, the deviation action is bounded and the coefficients of the state equation are regular. Hence, $(e_{t+\varepsilon}^a,W_{t+\varepsilon}^a)$ has finite moments of every order. Conditional on this state, the strategy agrees with the candidate feedback on $[t+\varepsilon,T)$, so Lemma~\ref{lem:risk_averse_candidate_moments} establishes local square-integrability, absolute integrability of cumulative trading, and all required polynomial wealth moments.

\end{comment}

It remains to verify the exponential moment. Since $a$ is constant and the pricing error is Gaussian on $[t,t+\varepsilon]$, the random variable $W_{t+\varepsilon}^a=w+a\int_t^{t+\varepsilon}e_s^a\,\mathrm{d}s$ is Gaussian and therefore has finite exponential moments of every order. After $t+\varepsilon$, the candidate feedback is followed and its contribution to wealth is nonnegative. Consequently,
\[
    W_T^{a\otimes_{t,\eps}\widehat\alpha^*  }
    \geq
    W_{t+\varepsilon}^a,
    \qquad
    \mathbb{E}_{t,e,w}
    \left[
        e^{-\rho W_T^{a\otimes_{t,\eps}\widehat\alpha^*  }}
    \right]
    \leq
    \mathbb{E}_{t,e,w}
    \left[
        e^{-\rho W_{t+\varepsilon}^a}
    \right]
    <\infty.
\]
Thus, $a\otimes_{t,\eps}\widehat\alpha^*  $ is admissible.
\end{proof}

\begin{comment}
\begin{proof}[Proof of Proposition~\ref{prop:risk_averse_admissibility}]
Lemma~\ref{lem:risk_averse_candidate_moments} establishes well-posedness of the state equation, convergence of the pricing error, absolute integrability of cumulative trading, and the polynomial wealth estimates. Lemma~\ref{lem:risk_averse_exponential_integrability} proves the additional exponential-integrability condition required when $\rho>0$. Therefore, $\widehat{\alpha}^*$ belongs to the admissible class introduced in Section~\ref{sec:model}. Proposition~\ref{prop:risk_averse_deviation_admissibility} further shows that the admissible class is stable under the bounded constant deviation deviations used below.
\end{proof}
\end{comment}

The only additional admissibility requirement created by risk aversion is therefore the exponential moment of terminal wealth. With this condition established, we may proceed to the verification of the entropic continuation value and the weak equilibrium condition.

% 1~remember to change 'h'
\subsubsection{Proof of the verification theorem}
\label{subsubsec:risk_averse_verification}

The verification follows the similar structure as in the risk-neutral case, with one additional step arising from the entropic certainty equivalent. We first identify the exponential continuation function and the conditional moments of terminal wealth, then verify rational pricing, and finally establish the weak equilibrium condition under bounded pasting deviations.

Throughout this subsubsection, $\widehat{\alpha}^*(t,e)=k(t)e$, where $\lambda$ and $k$ are the coefficient functions determined in Theorem~\ref{thm:risk_averse_coefficient_system} and \eqref{eq:risk_averse_A_B_k}. 
%In the calculations below, we write $\lambda$ and $k$ for $\lambda^*$ and $k^*$, respectively. 
Recall the candidate pricing rule is $H^*(t,\xi)=\mu+\xi$. Under the change of variable $e=v-H^*(t,\xi)=v-\mu-\xi$, the generator in \eqref{eq:insider_generator} becomes
\begin{equation}
    \mathcal{L}^a\varphi
    =
    \varphi_t
    -\lambda(t)a\varphi_e
    +ea\varphi_w
    +\frac{1}{2}\sigma^2\lambda(t)^2\varphi_{ee}.
    \label{eq:risk_averse_verification_generator}
\end{equation}

Define
\begin{equation}
    \Phi(t,e,w)
    :=
    e^{-\rho(w+h(t,e))},
    \qquad
    g(t,e,w)
    :=
    w+m(t,e),
    \qquad
    U(t,e,w)
    :=
    w+u(t,e),
    \label{eq:risk_averse_verification_functions}
\end{equation}
where $h$, $m$, and $u$ are defined in \eqref{eq:risk_averse_continuation_functions}, and introduce the second-moment candidate
\begin{equation}
    Q(t,e,w)
    :=
    g(t,e,w)^2
    +
    \frac{2}{\eta}
    \bigl(w+h(t,e)-U(t,e,w)\bigr).
    \label{eq:risk_averse_second_moment_Q}
\end{equation}
For a fixed $t<T$, let $(t_n)_n$ be an arbitrary deterministic sequence satisfying $t\leq t_n<T$ and $t_n\uparrow T$.

The polynomial estimates needed to pass to the terminal time are inherited from the risk-neutral verification. The presence of $h$ in \eqref{eq:risk_averse_second_moment_Q} does not increase the order of growth because $h$ is also quadratic in the pricing error.

\begin{proposition}
\label{prop:risk_averse_verification_uniform_integrability}
The candidate state satisfies $e_{t_n}\to0$ and $W_{t_n}^{\widehat{\alpha}^*}\to W_T^{\widehat{\alpha}^*}$ in $L^p$ for every $p\geq1$. Moreover, $\{g(t_n,e_{t_n},W_{t_n})\}_n$ and $\{Q(t_n,e_{t_n},W_{t_n})\}_n$ are uniformly integrable. For every $n$, the stochastic integrals in the It\^o decompositions of $g(s,e_s,W_s)$ and $Q(s,e_s,W_s)$ on $[t,t_n]$ are square-integrable martingales.
\end{proposition}

\begin{proof}
Lemma~\ref{lem:risk_averse_candidate_moments} establishes $e_{t_n}\to0$ in every $L^p$ and almost surely. Since $W_s^{\widehat{\alpha}^*}-w=\int_t^s k(r)e_r^2\,\mathrm{d}r$ is nondecreasing, \eqref{eq:risk_averse_proof_wealth_moments} and dominated convergence imply $W_{t_n}^{\widehat{\alpha}^*}\to W_T^{\widehat{\alpha}^*}$ in $L^p$ for every $p\geq1$.

The coefficients of $h$, $m$, and $u$ remain bounded on $[t,T]$. Consequently, there exist constants $C_1,C_2>0$, independent of $s\in[t,T)$, such that
\[
    |w+h(s,e)|+|g(s,e,w)|+|U(s,e,w)|
    \leq
    C_1(1+|w|+|e|^2),
    \qquad
    |Q(s,e,w)|
    \leq
    C_2(1+|w|^2+|e|^4).
\] % 3~check the expression C_n
The moment estimates in Lemma~\ref{lem:risk_averse_candidate_moments} therefore imply that $\{g(t_n,e_{t_n},W_{t_n})\}_n$ and $\{Q(t_n,e_{t_n},W_{t_n})\}_n$ are bounded in $L^2$ and hence uniformly integrable. Finally, $g_e$ grows linearly in $e$, while the explicit form of $Q$ implies $|Q_e(s,e,w)|\leq C_n(1+|w||e|+|e|^3)$ on $[t,t_n]$. The same argument used in Proposition~\ref{prop:risk_neutral_verification_uniform_integrability}, together with \eqref{eq:risk_averse_proof_wealth_moments} and the normalized pricing-error estimates in \eqref{eq:risk_averse_normalized_error_moments}, shows that
\[
    \mathbb{E}_{t,e,w}
    \left[
        \int_t^{t_n}
        \sigma^2\lambda(s)^2
        \bigl(g_e(s,e_s,W_s)^2+Q_e(s,e_s,W_s)^2\bigr)
        \,\mathrm{d}s
    \right]
    <\infty.
\]
Thus, the corresponding stochastic integrals are square-integrable martingales.
\end{proof}

We first identify the exponential continuation function. This is the additional step required by risk aversion.
%4~减少换行
\begin{lemma}
\label{lem:risk_averse_entropic_continuation}
The function $h$ defined in \eqref{eq:risk_averse_continuation_functions} satisfies
\begin{equation}
    w+h(t,e)
    =
    -\frac{1}{\rho}
    \log
    \mathbb{E}_{t,e,w}
    \left[
        e^{-\rho W_T^{\widehat{\alpha}^*}}
    \right].
    \label{eq:risk_averse_verification_h}
\end{equation}
\end{lemma}

\begin{proof}
The auxiliary equation \eqref{eq:auxiliary_h} implies $\mathcal{L}^{\widehat{\alpha}^*}\Phi=0$. Hence, It\^o's formula shows that $\Phi(s,e_s,W_s)$ is a local martingale. The bound \eqref{eq:risk_averse_entropic_process_bound} makes it a bounded martingale, so
\[
    \Phi(t,e,w)
    =
    \mathbb{E}_{t,e,w}
    \left[
        \Phi(t_n,e_{t_n},W_{t_n})
    \right].
\]
By Lemma~\ref{lem:risk_averse_candidate_moments}, $e_{t_n}\to0$ and $W_{t_n}\to W_T^{\widehat{\alpha}^*}$ almost surely. As $s\uparrow T$, Since $\Gamma(s)$ remains bounded and $E(s)\to0$, we also have $h(s,e_s)\to0$ almost surely. Therefore, $\Phi(t_n,e_{t_n},W_{t_n})\to e^{-\rho W_T^{\widehat{\alpha}^*}}$ almost surely. The bound \eqref{eq:risk_averse_entropic_process_bound} permits passage to the limit and yields
\[
    e^{-\rho(w+h(t,e))}
    =
    \mathbb{E}_{t,e,w}
    \left[
        e^{-\rho W_T^{\widehat{\alpha}^*}}
    \right].
\]
Taking logarithms proves \eqref{eq:risk_averse_verification_h}.
\end{proof}

The preceding proposition also permits the same terminal-time argument as in the risk-neutral case for the first and second moments of terminal wealth.

\begin{lemma}
\label{lem:risk_averse_continuation_moments}
The functions $g$ and $U$ defined in
\eqref{eq:risk_averse_verification_functions} satisfy
\begin{equation}
    g(t,e,w)
    =
    \mathbb{E}_{t,e,w}
    \left[
        W_T^{\widehat{\alpha}^*}
    \right]
    \label{eq:risk_averse_verification_g}
\end{equation}
and
\begin{equation}
    U(t,e,w)
    =
    -\frac{1}{\rho}
    \log
    \mathbb{E}_{t,e,w}
    \left[
        e^{-\rho W_T^{\widehat{\alpha}^*}}
    \right]
    -
    \frac{\eta}{2}
    \operatorname{Var}_{t,e,w}
    \left(
        W_T^{\widehat{\alpha}^*}
    \right).
    \label{eq:risk_averse_verification_U}
\end{equation}
\end{lemma}

\begin{proof}
Under the candidate strategy, the auxiliary equation
\eqref{eq:auxiliary_m} implies
$\mathcal{L}^{\widehat{\alpha}^*}g=0$. It\^o's formula and
Proposition~\ref{prop:risk_averse_verification_uniform_integrability}
yield
\[
    g(t,e,w)
    =
    \mathbb{E}_{t,e,w}
    \left[
        g(t_n,e_{t_n},W_{t_n})
    \right].
\]
Since $m(t_n,e_{t_n})\to0$ in $L^1$ and
$W_{t_n}\to W_T^{\widehat{\alpha}^*}$ in $L^1$, passage to the limit
proves \eqref{eq:risk_averse_verification_g}.

For the second moment, \eqref{eq:auxiliary_h},
\eqref{eq:auxiliary_m}, and \eqref{eq:reduced_extended_hjb} imply
\[
    \mathcal{L}^{\widehat{\alpha}^*}(w+h)
    =
    \frac{\rho}{2}\sigma^2\lambda(t)^2h_e^2,
    \qquad
    \mathcal{L}^{\widehat{\alpha}^*}U
    =
    \frac{\rho}{2}\sigma^2\lambda(t)^2h_e^2
    +
    \frac{\eta}{2}\sigma^2\lambda(t)^2m_e^2.
\]
Moreover,
$\mathcal{L}^{\widehat{\alpha}^*}(g^2)
=\sigma^2\lambda(t)^2m_e^2$. It follows from
\eqref{eq:risk_averse_second_moment_Q} that
$\mathcal{L}^{\widehat{\alpha}^*}Q=0$. Thus, It\^o's formula and
Proposition~\ref{prop:risk_averse_verification_uniform_integrability}
imply
\[
    Q(t,e,w)
    =
    \mathbb{E}_{t,e,w}
    \left[
        Q(t_n,e_{t_n},W_{t_n})
    \right].
\]

The terminal conditions for $h$, $m$, and $u$, together with
$e_{t_n}\to0$ and $W_{t_n}\to W_T^{\widehat{\alpha}^*}$, imply
$Q(t_n,e_{t_n},W_{t_n})
\to(W_T^{\widehat{\alpha}^*})^2$ in probability. Uniform integrability
then proves
\[
    Q(t,e,w)
    =
    \mathbb{E}_{t,e,w}
    \left[
        \bigl(W_T^{\widehat{\alpha}^*}\bigr)^2
    \right].
\]
Finally, substituting this identity together with
\eqref{eq:risk_averse_verification_h} and
\eqref{eq:risk_averse_verification_g} into
\eqref{eq:risk_averse_second_moment_Q} proves
\eqref{eq:risk_averse_verification_U}.
\end{proof}

We next verify rational pricing. The filtering argument is unchanged because it depends only on the linear feedback form and the filtering consistency relation.

\begin{lemma}
\label{lem:risk_averse_rational_candidate_price}
Let $P_t^*=H^*(t,\xi_t^*)=\mu+\xi_t^*$, where $\xi_t^*=\int_0^t\lambda(s)\,\mathrm{d}Y_s^*$ and $\mathrm{d}Y_t^*=\widehat{\alpha}_t^*\,\mathrm{d}t+\sigma\,\mathrm{d}B_t$. Then
\begin{equation}
    P_t^*
    =
    \mathbb{E}
    \left[
        V\mid\mathcal{F}_t^{Y^*}
    \right],
    \qquad
    \operatorname{Var}
    \left(
        V\mid\mathcal{F}_t^{Y^*}
    \right)
    =
    \Sigma(t),
    \qquad
    0\leq t<T.
    \label{eq:risk_averse_verified_rational_price}
\end{equation}
Moreover, $\mathbb{E}[\widehat{\alpha}_t^*\mid\mathcal{F}_t^{Y^*}]=0$ and $P_T^*=V$ almost surely.
\end{lemma}

\begin{proof}
The risk-averse coefficients satisfy the same filtering identities $k(t)\Sigma(t)/\sigma^2=\lambda(t)$ and $\Sigma'(t)=-k(t)^2\Sigma(t)^2/\sigma^2$ used in Lemma~\ref{lem:risk_neutral_rational_candidate_price}. The localized proof of that lemma therefore applies without
modification and establishes
\eqref{eq:risk_averse_verified_rational_price}. Rational pricing implies $\mathbb{E}[\widehat{\alpha}_t^*\mid\mathcal{F}_t^{Y^*}]=k(t)\mathbb{E}[V-P_t^*\mid\mathcal{F}_t^{Y^*}]=0$. Finally, Lemma~\ref{lem:risk_averse_candidate_moments} establishes $V-P_t^*\to0$ almost surely as $t\uparrow T$, and hence $P_T^*=V$ almost surely.
\end{proof}

It remains to verify the local equilibrium condition. The next proposition combines the exponential and mean--variance variations under a bounded constant pasting deviation.

\begin{proposition}
\label{prop:risk_averse_deviation_verification}
For every $t<T$, every state $(e,w)$, and every bounded constant action $a$,
\begin{equation}
    \lim_{\varepsilon\downarrow0}
    \frac{
        J(t,e,w;a\otimes_{t,\eps}\widehat\alpha^* )
        -
        J(t,e,w;\widehat{\alpha}^*)
    }{\varepsilon}
    =0.
    \label{eq:risk_averse_verified_deviation_condition}
\end{equation}
Consequently, $\widehat{\alpha}^*$ is a weak equilibrium strategy.
\end{proposition}

\begin{proof}
%Let $a\otimes_{t,\eps}\widehat\alpha^*  $ coincide with $a$ on $[t,t+\varepsilon)$ and with $\widehat{\alpha}^*$ thereafter, and
Denote the state at $t+\varepsilon$ by
$(e_{t+\varepsilon}^a,W_{t+\varepsilon}^a)$.
Proposition~\ref{prop:risk_averse_admissibility} ensures that
this strategy is admissible for sufficiently small
$\varepsilon>0$. Since the candidate strategy is followed after
$t+\varepsilon$,
Lemmas~\ref{lem:risk_averse_entropic_continuation} and
\ref{lem:risk_averse_continuation_moments} imply
\[
    \begin{aligned}
        \mathbb{E}_{t,e,w}
        \left[
            e^{-\rho W_T^{a\otimes_{t,\eps}\widehat\alpha^*  }}
        \right]
        &=
        \mathbb{E}_{t,e,w}
        \left[
            \Phi(t+\varepsilon,e_{t+\varepsilon}^a,
            W_{t+\varepsilon}^a)
        \right],\\
        \mathbb{E}_{t,e,w}
        \left[
            W_T^{a\otimes_{t,\eps}\widehat\alpha^*  }
        \right]
        &=
        \mathbb{E}_{t,e,w}
        \left[
            g(t+\varepsilon,e_{t+\varepsilon}^a,
            W_{t+\varepsilon}^a)
        \right],\\
        \mathbb{E}_{t,e,w}
        \left[
            \bigl(W_T^{a\otimes_{t,\eps}\widehat\alpha^*  }\bigr)^2
        \right]
        &=
        \mathbb{E}_{t,e,w}
        \left[
            Q(t+\varepsilon,e_{t+\varepsilon}^a,
            W_{t+\varepsilon}^a)
        \right].
    \end{aligned}
\]
Let $\Delta_\Phi(\varepsilon)$, $\Delta_g(\varepsilon)$, and
$\Delta_Q(\varepsilon)$ denote the differences between the three
expectations above and $\Phi(t,e,w)$, $g(t,e,w)$, and $Q(t,e,w)$,
respectively.

Fix $\varepsilon_0>0$ such that $t+\varepsilon_0<T$, and let
$(e_s^a,W_s^a)_{t\leq s\leq t+\varepsilon_0}$ denote the state
process generated by the constant action $a$ on this interval. For
every $0<\varepsilon\leq\varepsilon_0$, this process agrees with
the state under $a\otimes_{t,\eps}\widehat\alpha^*  $ on
$[t,t+\varepsilon]$. Moreover,
\[
    e_s^a
    =
    e-a\int_t^s\lambda(r)\,\mathrm{d}r
    -\sigma\int_t^s\lambda(r)\,\mathrm{d}B_r,
    \qquad
    W_s^a
    =
    w+a\int_t^s e_r^a\,\mathrm{d}r,
    \qquad
    t\leq s\leq t+\varepsilon_0.
\]
The explicit representations above and the boundedness of
$\lambda$ on $[t,t+\varepsilon_0]$ imply that $e_s^a$ and
$W_s^a$ are Gaussian for each $s$ in this interval, with their
means and variances uniformly bounded in $s$. Consequently, the
polynomial moments of $e_s^a$ and the exponential moments of
$W_s^a$ are uniformly bounded. For every $q\geq1$ and
$c\in\mathbb{R}$, the Cauchy--Schwarz inequality therefore yields
\[
    \mathbb{E}_{t,e,w}\bigl[(1+|e_s^a|^q)e^{cW_s^a}\bigr]
    \leq
    \bigl(\mathbb{E}_{t,e,w}[(1+|e_s^a|^q)^2]\bigr)^{1/2}
    \bigl(\mathbb{E}_{t,e,w}[e^{2cW_s^a}]\bigr)^{1/2},
\]
and hence
\begin{equation}
    \sup_{t\leq s\leq t+\varepsilon_0}
    \mathbb{E}_{t,e,w}
    \left[
        (1+|e_s^a|^q)e^{cW_s^a}
    \right]
    <\infty.
    \label{eq:risk_averse_deviation_mixed_moment_bound}
\end{equation}

Since $h(s,e)=\Gamma(s)e^2+E(s)$ with $\Gamma,E\geq0$, and the
coefficients of $h$ and their derivatives are bounded on
$[t,t+\varepsilon_0]$, there exists $C>0$ such that
\[
    \left|
        \mathcal{L}^a\Phi(s,e,w)
    \right|
    \leq
    C(1+|e|^2)e^{-\rho w},
    \qquad
    \left|
        \sigma\lambda(s)\Phi_e(s,e,w)
    \right|^2
    \leq
    C|e|^2e^{-2\rho w}.
\]
Taking $q=2$ and $c=-2\rho$ in
\eqref{eq:risk_averse_deviation_mixed_moment_bound} gives
\[
    \mathbb{E}_{t,e,w}\!\left[\int_t^{t+\varepsilon_0}
    |\sigma\lambda(s)\Phi_e(s,e_s^a,W_s^a)|^2\,\mathrm{d}s\right]
    \leq C\varepsilon_0
    \sup_{t\leq s\leq t+\varepsilon_0}
    \mathbb{E}_{t,e,w}\!\left[|e_s^a|^2e^{-2\rho W_s^a}\right]
    <\infty.
\]
Thus, the stochastic integral in the It\^o decomposition of
$\Phi(s,e_s^a,W_s^a)$ is a square-integrable martingale.

Moreover, for any $p>1$,
\[
    \left|
        \mathcal{L}^a\Phi(s,e_s^a,W_s^a)
    \right|^p
    \leq
    C(1+|e_s^a|^{2p})e^{-p\rho W_s^a}.
\]
Taking $q=2p$ and $c=-p\rho$ in
\eqref{eq:risk_averse_deviation_mixed_moment_bound} shows that
\[
    \sup_{t\leq s\leq t+\varepsilon_0}
    \mathbb{E}_{t,e,w}
    \left[
        \left|
            \mathcal{L}^a\Phi(s,e_s^a,W_s^a)
        \right|^p
    \right]
    <\infty.
\]
Hence,
$\{\mathcal{L}^a\Phi(s,e_s^a,W_s^a):
t\leq s\leq t+\varepsilon_0\}$ is uniformly integrable.

The functions $g$, $Q$, and their derivatives have polynomial
growth in $(e,w)$, uniformly for
$s\in[t,t+\varepsilon_0]$. The uniform Gaussian moment bounds
therefore imply that the stochastic integrals in their It\^o
decompositions are square-integrable martingales and that
$\{\mathcal{L}^ag(s,e_s^a,W_s^a)\}$ and
$\{\mathcal{L}^aQ(s,e_s^a,W_s^a)\}$ are uniformly integrable.

It\^o's formula and the martingale property of the stochastic
integrals now yield, for $f\in\{\Phi,g,Q\}$,
\[
    \Delta_f(\varepsilon)
    =
    \int_t^{t+\varepsilon}
    \mathbb{E}_{t,e,w}
    \left[
        \mathcal{L}^af(s,e_s^a,W_s^a)
    \right]
    \,\mathrm{d}s.
\]
Since $(e_s^a,W_s^a)\to(e,w)$ almost surely as $s\downarrow t$ and
the generators are continuous,
\[
    \mathcal{L}^af(s,e_s^a,W_s^a)
    \to 
    \mathcal{L}^af(t,e,w)
    \qquad\text{almost surely}.
\]
The uniform integrability established above implies that this
convergence also holds in $L^1$. Consequently,
\[
    \left|\frac{\Delta_f(\varepsilon)}{\varepsilon}-\mathcal{L}^af(t,e,w)\right|
    \leq\frac{1}{\varepsilon}\int_t^{t+\varepsilon}
    \mathbb{E}_{t,e,w}\!\left[
    \left|\mathcal{L}^af(s,e_s^a,W_s^a)-\mathcal{L}^af(t,e,w)\right|
    \right]\mathrm{d}s
    \to 0.
\]
It follows that
\[
    \frac{\Delta_\Phi(\varepsilon)}{\varepsilon}
    \to 
    \mathcal{L}^a\Phi(t,e,w),
    \qquad
    \frac{\Delta_g(\varepsilon)}{\varepsilon}
    \to 
    \mathcal{L}^ag(t,e,w),
    \qquad
    \frac{\Delta_Q(\varepsilon)}{\varepsilon}
    \to 
    \mathcal{L}^aQ(t,e,w).
\]
In particular, $\Delta_\Phi(\varepsilon)=O(\varepsilon)$,
$\Delta_g(\varepsilon)=O(\varepsilon)$, and
$\Delta_Q(\varepsilon)=O(\varepsilon)$ as $\varepsilon\downarrow0$. Using \eqref{eq:risk_averse_verification_U}, the identity $U=-\rho^{-1}\log\Phi-\frac{\eta}{2}(Q-g^2)$, and $\log(1+x)=x+o(x)$, we obtain
\begin{equation}
    \begin{aligned}
        &\lim_{\varepsilon\downarrow0}
        \frac{
            J(t,e,w;a\otimes_{t,\eps}\widehat\alpha^*  )
            -U(t,e,w)
        }{\varepsilon}\\
        &\qquad=
        -\frac{\mathcal{L}^a\Phi(t,e,w)}{\rho\Phi(t,e,w)}
        -
        \frac{\eta}{2}
        \left(
            \mathcal{L}^aQ(t,e,w)
            -2g(t,e,w)\mathcal{L}^ag(t,e,w)
        \right).
    \end{aligned}
    \label{eq:risk_averse_deviation_first_variation}
\end{equation}
From \eqref{eq:risk_averse_verification_functions} and \eqref{eq:risk_averse_second_moment_Q}, we have
$$
\begin{aligned}
-\frac{\mathcal{L}^a\Phi}{\rho\Phi}
=&\mathcal{L}^a(w+h)-\frac{\rho}{2}\sigma^2\lambda(t)^2h_e^2\\
\mathcal{L}^aQ-2g\mathcal{L}^ag
=&\sigma^2\lambda(t)^2m_e^2
+\frac{2}{\eta} \left(\mathcal{L}^a(w+h)-\mathcal{L}^aU\right).
\end{aligned}
$$
Substituting these identities into
\eqref{eq:risk_averse_deviation_first_variation} yields
\[
    u_t
    +
    \frac{1}{2}\sigma^2\lambda(t)^2u_{ee}
    -
    \frac{\rho}{2}\sigma^2\lambda(t)^2h_e^2
    -
    \frac{\eta}{2}\sigma^2\lambda(t)^2m_e^2
    +
    a\bigl(e-\lambda(t)u_e\bigr).
\]
The reduced extended HJB system
\eqref{eq:reduced_extended_hjb} shows that this expression is zero
for every bounded constant action $a$. 
Lemma~\ref{lem:risk_averse_continuation_moments} also implies $U(t,e,w)=J(t,e,w;\widehat{\alpha}^*)$. This proves \eqref{eq:risk_averse_verified_deviation_condition}, and the weak equilibrium inequality follows immediately.
\end{proof}

\begin{proof}[Proof of Theorem~\ref{thm:risk_averse_main_equilibrium}]
Lemmas~\ref{lem:risk_averse_entropic_continuation} and \ref{lem:risk_averse_continuation_moments} establish the three identities in \eqref{eq:risk_averse_value_representation}. Proposition~\ref{prop:risk_averse_deviation_verification} verifies the informed trader's weak equilibrium condition, while Lemma~\ref{lem:risk_averse_rational_candidate_price} establishes rational pricing, inconspicuous trading, and full revelation. Therefore, $\bigl((H^*,\lambda),\widehat{\alpha}^*\bigr)$ is a weak market equilibrium.
\end{proof}

The verification completes the construction of the risk-averse Gaussian--linear equilibrium. In particular, the exponential continuation function, the first two conditional moments of terminal wealth, the local equilibrium condition, and the rational pricing condition are mutually consistent with the candidate coefficients obtained from the shooting problem.

\noindent{\bf Acknowledgements:}
Shuoqing Deng is supported by the Hong Kong University of Science and Technology Start-up Grant No. R9826 and Hong Kong RGC Early Career Scheme (ECS) Grant No. 26307125.
Zhenhua Wang is supported by Shandong Excellent Young Scientists Fund Program (Overseas) under grant No. 2025HWYQ–022 and National Natural Science Foundation of China (NSFC) under grant No.12501659.

%\appendix

%\section{Appendix}

%aaa

\end{document}